\documentclass[11pt,leqno]{article}
\usepackage{amsfonts,amssymb,amsmath,amsthm,latexsym}
\usepackage{amsmath}
\usepackage{color,graphicx}
\newtheorem{theorem}{Theorem}

\newtheorem{corollary}[theorem]{Corollary}

\newtheorem{lemma}[theorem]{Lemma}

\newtheorem{proposition}[theorem]{Proposition}
\newtheorem{remark}[theorem]{Remark}

\newcommand{\wfin}{ ${\textstyle \Box}$    \vspace{3mm} }

\title{A mixed finite elements approximation of inverse source problems for the wave equation with variable coefficients using observability}
\author{Carlos \textsc{Castro} \thanks{Departamento de Matem\'atica e Inform\'atica Aplicadas a la Ingeniería Civil y Naval, Universidad Polit\'ecnica de Madrid, Madrid, 28040, Spain. E-mail: {\tt carlos.castro@upm.es}}
\and 
Sorin \textsc{Micu}\thanks{Department of Mathematics, University of
Craiova, 200585 and Gheorghe Mihoc-Caius Iacob Institute of Mathematical Statistics and Applied Mathematics of the Romanian Academy, 050711, Bucharest, Romania. E-mail: {\tt sd$\_$micu@yahoo.com}.} 
}

\begin{document}

\maketitle

\abstract{We consider an inverse problem for the linear one-dimensional wave equation with variable coefficients consisting in determining an unknown source term from a boundary observation. A method to obtain approximations of this inverse problem using a space discretization based on a mixed finite element method is proposed and analyzed.
Its stability and convergence properties relay on a new uniform boundary observability inequality with respect to the discretization parameter. This fundamental uniformity property of the observation is not verified in the case of more usual discretization schemes, such as centered finite differences or classical finite elements. For the mixed finite elements it is proved by combining lateral energy estimates with Fourier techniques.
}

\


\


\section{Introduction}

We consider the one dimensional wave equation with a given potential $a\in L^\infty(0,1)$ and homogeneous Dirichlet boundary conditions,
\begin{equation}\label{eq.wave}
\left\{ \begin{array}{ll} u''(t,x)-u_{xx}(t,x)+a(x)u(t,x)= 0&
t\in(0,T),\,\, x\in(0,1)\\ u(t,0)=u(t,1)=0\,\, & t\in(0,T)\\
u(0,x)=u^0(x),\,\, u'(0,x)=u^1(x)& x\in(0,1),\end{array}\right.
\end{equation}
where $\begin{pmatrix}u^0\\u^1\end{pmatrix}$ are the initial data. For such low regularity potential function, the following observability inequality is proved in \cite{Z}: given $T>2$, there exists a constant $C$ such that 
\begin{equation} \label{eq_observab}
\int_0^T |u_x(t,0)|^2 \,{\rm d}t \geq C\left\| \begin{pmatrix}u^0\\u^1\end{pmatrix} \right\|_{H^1_0 \times L^2}^2\qquad \left( \begin{pmatrix}u^0\\u^1\end{pmatrix}\in H^1_0(0,1)\times L^2(0,1)\right).
\end{equation}
The above relation roughly states that the norm of the initial data can be estimated from the solution at one extremity of the interval $(0,1)$ for sufficiently large time $T$. 

As remarked in \cite{PI}, the observability inequality \eqref{eq_observab} can be used to study the stability property of some inverse problems (see also  \cite{ASTT, ACT2, IY,KY}). To illustrate this application, we consider the following inverse source problem: given $T>0$, $\begin{pmatrix}w^0\\w^1\end{pmatrix}\in H^2 (0,1)\cap H^1_0(0,1)\times H^1_0(0,1)$ and $\lambda\in H^1(0,T)$ such that $\lambda(0)\neq 0$ determine the function $f \in L^2(0,1)$ from the observation $w_x'(t,0)$  in the time interval $(0,T)$ of the solution to the following equation
\begin{equation}\label{eq:inv_cont0}
\left\{ \begin{array}{ll} w''(t,x)-w_{xx}(t,x)+a(x)w(t,x)= \lambda(t)f(x)&
t\in(0,T),\,\, x\in(0,1)\\ w(t,0)= w(t,1)=0\,\, & t\in(0,T)\\
w(0,x)=w^0(x),\,\, w'(0,x)=w^1(x)& x\in(0,1).\end{array}\right.
\end{equation}
Estimate \eqref{eq_observab} implies that the following stability result holds if $T>2$: there exists a constant $C>0$ such that 
\begin{equation}\label{eq:stabcont}
\left\|\widehat{f}-\widetilde{f}\right\|_{L^2(0,1)}\leq C \left\|\widehat{w}_x'(\,\cdot\,, 0)-\widetilde{w}_x'(\,\cdot\,, 0)\right\|_{L^2(0,T)},
\end{equation}
for any $\widehat{f},\,\widetilde{f}\in L^2(0,1)$, where $\widehat{w},\,\widetilde{w}$ are the corresponding solutions of \eqref{eq:inv_cont0} with non-homo\-ge\-neous terms $\lambda \widehat{f}$ and  $\lambda \widetilde{f}$, respectively. Inequality \eqref{eq:stabcont} shows, in particular, that the unknown source $f$ can be uniquely determined by the observation term $w_x'$ if the time $T$ is long enough and it is fundamental in any reconstruction method of $f$.

The aim of this paper is to give a convergent algorithm to find the numerical approximation of the source term $f$ when the boundary observation $w_x'(t,0)$ is known in the time interval $(0,T)$. The basic idea, often used in this context, is to minimize the functional ${\mathcal J}:L^2(0,T)\rightarrow \mathbb{R}$, defined by 
\begin{equation}\label{eq:mincont}
{\mathcal J}(\widehat{f})=\frac{1}{2}\left\| \widehat{w}_x'(\,\cdot\,,0)-w_x'(\,\cdot\,,0)\right\|^2_{L^2(0,T)},
\end{equation} 
where $\widehat{w}$ is the solution of  \eqref{eq:inv_cont0} with non-homo\-ge\-neous term $\lambda \widehat{f}$. Note that the unknown source term $f$ is a minimizer of the functional $\mathcal J$. Moreover, the stability condition \eqref{eq:stabcont} ensures the uniqueness of the minimizer of $\mathcal J$. 

When we are interested in the numerical approximations for inverse problems a discrete version of the functional $\mathcal J$ must be considered. This requires in particular to substitute the continuous equation \eqref{eq:inv_cont0} by a consistent discrete approximation in the space variable, depending on a discretization parameter $h$ tending to zero. To show the associated stability result for the approximate system a discrete version of  \eqref{eq_observab} is needed. Naturally, in the discrete context, the constant $C$ in the right-hand member of \eqref{eq_observab} will depend on the discretization parameter $h$ and will be denoted in the following by $C_h$. It is important to find such discrete version of (\ref{eq_observab}) with a constant $C_h$ uniformly bounded from below in $h$. Indeed, this condition is essential in proving the convergence of the reconstruction algorithm. 

It turns out that this is not the case when considering usual discretization of the wave equation based on finite differences or classical finite elements. This has been observed by a number of authors in different situations (see \cite{EZb} and the references therein for a complete description of this phenomenon and  some concrete examples). More precisely, in these cases the constant $C_h$ in the associated discrete inverse inequality tends to zero as $h\to 0$. 

There are several cures proposed in the literature to deal with this lack of uniformity of the constant $C_h$ with respect to $h$. Among the techniques used in this context we mention Tychonoff regularization \cite{GLL,GL}, bi-grid \cite{IgZ}, space-time methods \cite{NC1,NC2,NC3}, filtering techniques \cite{IZ,micu,LR}, vanishing viscosity method \cite{micu_siam}, etc. We also mention the recent volumes \cite{NumCont} for an overview of the state of the art in the application of these and other related methods.

In \cite{CM,CMM} a numerical method based on a mixed finite element formulation for the wave equation, firstly proposed in \cite{bi} for  stabilization problems, was analyzed. It was shown that the  corresponding semi-discrete version of the observability inequality  holds with a constant $C_h$ uniformly bounded in $h$. This was proved for the constant coefficients wave equation in one dimension and two dimension in a square.  The proof in the one-dimensional case is based on a Fourier series argument and requires a detailed spectral analysis, while the two-dimensional case relies on a discrete version of the classical multipliers method.  In this paper we extend the Fourier series proof in \cite{CM} to the wave equation  with a nonconstant potential. The main difficulty is that, in this case, we do not have an explicit formula for the eigenvalues and eigenfunctions of  neither the continuous or the discrete problem. The basic idea of the proof is to apply Ingham's inequality by showing that the corresponding sequence of discrete eigenvalues have a uniform gap. This and other additionally required spectral properties are obtained by a novel approach combining a contradiction argument with the lateral energy technique used in \cite{Z} to obtain the observability inequality for the continuous case. The precise result is presented in Theorem \ref{te.obsineg} stated in the following section and its complete proof is given in Section \ref{sec:4} below.

The uniform observability estimate shown in Theorem \ref{te.obsineg}  allows us to propose a convergent numerical algorithm for the reconstruction of the source term $f$ in problem  \eqref{eq:inv_cont0} based on a  discretization using the mixed finite elements method combined with the least-squares procedure \eqref{eq:mincont}. The description of the corresponding discrete inverse problem and the results of stability and convergence of the approximating scheme are given in Section \ref{sec:2}, Theorems \ref{te:estab0} and \ref{te:coninv}, respectively. 

Let us mention that other approximation methods for inverse problems and their relation with suitable discrete uniform observability inequalities have been studied in the literature. For instance, in \cite{BE,BBE,BEO} the problem of recovering a potential in the wave equation from the knowledge of the boundary data is analyzed.  However, instead of the mixed finite elemements, the classical finite difference scheme is considered. We have already mentioned that, for this particular numerical scheme, the natural observability inequality is not uniform. Therefore, to restore the uniformity, an extra observation term in the entire spatial domain, vanishing in the limit, is introduced. As mentioned in \cite{BE,BBE,BEO}, this corresponds to a Tychonoff regularization procedure. Let us mention that, in the case of the scheme studied by us, two discrete observation terms appear due to the corresponding non-diagonal mass matrix. One of them converges to the continuous observation term while the other one vanishes when the discretization step $h$ goes to zero. However, unlike the situation analyzed in \cite{BE,BBE,BEO}, the vanishing term is also located at the extremity from which the observation is made.

The reconstruction of a source term from a partial boundary observation is also considered in \cite{NC1,NC2} by a  least-squares technique as in this article. However, unlike in our approach,  \cite{NC1,NC2} deduce an optimality condition by taking the hyperbolic equation as the main constraint of the problem. The reconstruction is based on a numerical approximation by a space-time finite element discretization of this optimality condition which consists on an elliptic formulation in both variables. A related least-squares approach is used in \cite{MT} to approximate the controls for a nonlinear wave equation.

A different method, also related to the observability property but in a less direct way, is used in \cite{HR,RTW} in order to recover the initial state of an infinite-dimensional system from a known output function on a given time interval. By using the uniform observability estimate proved for mixed finite elements in Theorem \ref{te.obsineg} below,  the minimization  procedure from our paper can be adapted to solve this problem, too. In contrast to this, the method studied in \cite{HR,RTW}  consists on an   iterative algorithm based on Russell's principle which asserts that, for an operator group, forward and backward stabilizability implies exact controllability.

The rest of the paper is divided as follows. In Section 2 we introduce the mixed finite element discretization of (\ref{eq.wave}) and \eqref{eq:inv_cont0}, stating the main results. In Section 3 we prove some spectral properties for the associated discrete operator. Section 4 is devoted to the main observability result, uniform with respect to the discretization parameter. In Sections 5 we show the stability and convergence results concerning the inverse source problem.  We present some numerical experiments in the following section. Finally, the Appendix contains convergence results in the energy space for  the mixed finite elements scheme used in the proofs.

\section{A mixed finite element method and main results}\label{sec:2}

We first introduce a space semi-discretization using a mixed finite element method of the following more general  one-dimensional linear wave equation:
\begin{equation}\label{eq.wave_ap1}
\left\{ \begin{array}{ll} w''(t,x)-w_{xx}(t,x)+a(x)w(t,x)= g(t,x)&
t\in(0,T),\,\, x\in(0,1)\\ w(t,0)=w(t,1)=0\,\, & t\in(0,T)\\
w(0,x)=w^0(x),\,\, w'(0,x)=w^1(x)& x\in(0,1).\end{array}\right.
\end{equation} Let us consider $N\in \mathbb{N}^\ast$, $h=\frac{1}{N+1}$ and an uniform grid of the
interval $(0,1)$ given by $0=x_0<x_1<...<x_N<x_{N+1}=1$, with
$x_j=jh$, $0\leq j\leq N+1$. We use two families of basis functions $\{\varphi_j\}_{j=1}^N$ and $\{\psi_j\}_{j=1}^N$ where $\varphi_j,\psi_j:[0,1]\rightarrow
\mathbb{R}$ are given by
$$\varphi_j(x)=\left\{\begin{array}{cc} \frac{x-x_{j-1}}{h} & \mbox{ if
} x\in [x_{j-1},x_j] \\\frac{x_{j+1}-x}{h} & \mbox{ if } x\in
[x_{j},x_{j+1}] \\ 0 & \mbox{ otherwise} ,\end{array}\right.\quad
\psi_j(x)=\left\{\begin{array}{cc} \frac{1}{2} & \mbox{ if } x\in
(x_{j-1},x_{j+1}) \\ 0 & \mbox{ otherwise.} \end{array}\right.
$$
The main idea in a mixed finite element method is to approximate each component of the solution $\begin{pmatrix}w\\w'\end{pmatrix}$ to  \eqref{eq.wave_ap1} by linear combinations of the two different basis functions above as follows:   
\begin{equation}\label{eq.approx}
 w(t,x)\approx w_h(t,x):=\sum_{j=1}^{N} w_j(t) \varphi_j(x) ,\qquad 
 w'(t,x)\approx v_h(t,x):=\sum_{j=1}^{N} w'_j(t) \psi_j(x),
\end{equation}
for some coefficients $w_j(t)$ (see, for instance, \cite{bf}). Accordingly, the initial data $\begin{pmatrix}w^0\\w^1\end{pmatrix}$ is discretized as follows:
\begin{equation}\label{eq.approxinitdata}
 w^0(x)\approx w_h^0(x):=\sum_{j=1}^{N} w_j^0 \varphi_j(x) ,\qquad 
 w^1(x)\approx w_h^1(x):=\sum_{j=1}^{N} w_j^1 \psi_j(x).
\end{equation}
Observe that in (\ref{eq.approx}), the classical linear splines approximate the position $w$, whereas discontinuous piecewise constant functions are used for the velocity $w'$. Moreover, we considered a discrete nonhomogeneous term $g_h$ given by 
$$g(t,x)\approx g_h(t,x):=\sum_{j=1}^N g_j(t)\psi_j(x).$$
System (\ref{eq.wave}) is discretized in the following way:
\begin{equation}\label{eq.waverew1}
\left\{\begin{array}{lll}
\displaystyle \frac{d}{\,{\rm d}t}\int_0^1 w_h(t,x) \psi_j(x) \,{\rm d}x = \int_0^1 v_h(t,x)
\psi_j(x) \,{\rm d}x,  \mbox{ for all }1\leq j\leq
N,\\
\\  \displaystyle \frac{d}{\,{\rm d}t}<v_h(t,\,\cdot\,), \varphi_j>_{H^{-1},H^1_0} =
\int_0^1
(w_h)_x(t,x) (\varphi_j)_x(x) \,{\rm d}x - \\  \qquad \qquad  \displaystyle \int_0^1
a(x)w_h(t,x) \varphi_j(x) \,{\rm d}x +\int_0^1
g_h(t,x) \varphi_j (x) \,{\rm d}x , \mbox{ for all }1\leq j\leq N, \\ \\ 
w_h(0,x)=w^0_h(x),\,\, v_h(0,x)=w^1_h(x), \mbox{ for }x\in (0,1).
\end{array}
\right.
\end{equation}

By taking into account that, for any $1\leq i,j\leq N$,
$$<\psi_i, \varphi_j>_{H^{-1},H^1_0}=\int_0^1 \psi_i (x) \varphi_j(x) \,{\rm d}x
=\int_0^1 \psi_i (x) \psi_j(x) \,{\rm d}x =\left\{\begin{array}{ll}
\frac{h}{2} &\mbox{ if }i=j \\
\frac{h}{4} &\mbox{ if }|i-j|=1 \\
0 &\mbox{ othewise, }\end{array} \right.$$ and $$\int_0^1
(\varphi_i)_x (x) (\varphi_j)_x(x) \,{\rm d}x =\left\{\begin{array}{ll}
\frac{2}{h} &\mbox{ if }i=j \\
-\frac{1}{h} &\mbox{ if }|i-j|=1 \\
0 &\mbox{ othewise, }\end{array} \right.$$ 
and by considering an approximation of the potential term with a trapezoidal scheme 
\begin{equation}\label{eq:trap}\int_0^1 a(x) \varphi_i(x) \varphi_j(x)\approx \left\{
\begin{array}{ll}  0 & \mbox{ if }i\neq j \\
h a(x_j):=h a_j & \mbox{ if }i=j,\end{array}\right.\end{equation}
we obtain that (\ref{eq.waverew1}) is equivalent with the following system 
\begin{equation}\label{eq.dis0} \left\{\begin{array}{ll}
\vspace{2mm}

\displaystyle
\frac{h}{4}\left[2w''_j(t)+w''_{j+1}(t)+w''_{j-1}(t)\right]\\

\vspace{2mm}

\displaystyle
+\frac{1}{h}\left[2w_j(t)-w_{j+1}(t)-w_{j-1}(t)\right]+h a_j
w_j(t)
 =\widetilde{g}_j(t),&\mbox{ for }1\leq j\leq N,\,\,\, t>0,\\
w_0(t)=0,\,\, w_{N+1}(t)=0, &\mbox{ for }t>0,\\
w_j(0)=w^0_j,\,\, w'_j(0)=w^1_j,&\mbox{ for }1\leq j\leq N,
\end{array}
\right.
\end{equation}
where $\widetilde{g}_j(t)=\frac{h}{4}\left(2g_j(t)+g_{j+1}(t)+g_{j-1}(t) \right)$, $1\leq j\leq N$.

System (\ref{eq.dis0}) consists of $N$ linear differential
equations with $N$ unknowns $w_1, w_2,\break \ldots,w_N$.  In \eqref{eq.dis0} and in the sequel, given any $W=\left[w_j\right]_{1\leq j\leq N}\in \mathbb{C}^{N}$, by convention,  we consider that  $w_0=w_{N+1}=0$. It is convenient to write system (\ref{eq.dis0}) in an equivalent matrix form. Consider the matrices $ K_h,M_h,L_h \in {\cal M}_{N\times N}(\mathbb{R})$ and the vectors $G_h(t),W_h(t)\in \mathbb{C}^N$ defined as follows: {\small $$K_h=\frac{1}{h}
\left(\begin{array}{ccccccc}
2&-1&0&0&...&0&0\\
-1&2&-1&0&...&0&0\\
0&-1&2&-1&...&0&0\\
...&...&...&...&...&...&...\\
0&0&0&0&...&2&-1\\
0&0&0&0&...&-1&2
\end{array}\right),\, M_h=\frac{h}{4}
\left(\begin{array}{ccccccc}
2&1&0&0&...&0&0\\
1&2&1&0&...&0&0\\
0&1&2&1&...&0&0\\
...&...&...&...&...&...&...\\
0&0&0&0&...&2&1\\
0&0&0&0&...&1&2
\end{array}\right),$$
$$
L_h=h\left(\begin{array}{ccccccc}
a_1&0&0&0&.....&0&0\\
0&a_2&0&0&.....&0&0\\
0&0&a_3&0&.....&0&0\\
...&...&...&...&.....&...&...\\
0&0&0&0&.....&a_{N-1}&0\\
0&0&0&0&.....&0&a_N
\end{array}\right), \, G_h= \left(\begin{array}{c}
g_1\\
g_2\\
g_3\\
...\\
g_{N-1}\\
g_N
\end{array}\right) , \,W_h= \left(\begin{array}{c}
w_1\\
w_2\\
w_3\\
...\\
w_{N-1}\\
w_N
\end{array}\right) .
$$
}

Note that the components $\left[g_j\right]_{1\leq j\leq N}$ and $\left[w_j\right]_{1\leq j\leq N}$ of the vectors $G_h$ and $W_h$, respectively, depend on $h$ but, in order to keep the notation as simple as possible, we have chosen to not make explicit this dependence.  The semi-discrete system (\ref{eq.dis0}) can be written as
\begin{equation}\label{eq.general}
\left\{\begin{array}{ll}
M_h W_h''(t)+K_h W_h(t)+L_h W_h(t)=M_h G_h(t)\qquad (t\in (0,T))\\ 
W_h(0)=W_h^0,\,\, W'_h(0)=W_h^1.
\end{array}
\right.
\end{equation}
The unknown of (\ref{eq.general}) is the vector-valued function
$t\to W_h(t)\in\mathbb{C}^N$.  In the Appendix we study the convergence of  the solution of \eqref{eq.general} to the solution of the continuous wave equation \eqref{eq.wave_ap1}, under the hypothesis $a\in C[0,1]$, $\begin{pmatrix}w^0\\w^1\end{pmatrix} \in H_0^1(0,1)\times L^2(0,1)$ and $f\in L^2(0,T;L^2(0,1))$.

To state the main results we consider, besides the canonical inner product $\langle \,\cdot\,,\,\cdot\,\rangle $ (with its corresponding norm $\|\,\cdot\,\|$) defined in $\mathbb{C}^N$, two additional ones given by
\begin{equation}\label{eq.inprod2}
\displaystyle \langle U,W\rangle_1=\langle K_h U,W\rangle +\langle L_h
U,W\rangle,
\end{equation}
\begin{equation}\label{eq.inprodM} \displaystyle \langle
U,W\rangle_M=\langle M_h U,W\rangle,
\end{equation}
where $U,\,W\in\mathbb{C}^{N}$. Note that (\ref{eq.inprod2}) defines an inner product only if the matrix $K_h+L_h$ is positive defined. This is true when the diagonal values $(a_k)_{1\leq k\leq N}$ of the matrix $L_h$ are nonnegative or have sufficiently small absolute values. In order to avoid technical details, we will assume that this is the case by considering nonnegative potential functions 
\begin{equation}
\label{eq:pot}
a\in L^\infty(0,1),\quad a(x)\geq 0\qquad (x\in [0,1]).
\end{equation}
The corresponding norms to the inner products \eqref{eq.inprod2} and \eqref{eq.inprodM} will be denoted by $||\,\, \cdot \,\,||_1$ and $||\,\, \cdot \,\,||_M$, respectively.

Our first main result concerns the uniform observability of  the discrete version of equation \eqref{eq.wave} based on the mixed finite elements method described above.

\begin{theorem}\label{te.obsineg} Assume that \eqref{eq:pot} holds.  There exist two constants $T_0,\, \kappa_0>0$, independent of $h$, such that, for any $T> T_0$, the following inequality is verified:
\begin{equation}\label{eq:obsineq}
\int_0^{T} \left(\left|\frac{u_{1}(t)}{h}\right|^2+ \left|\frac{u_{1}'(t)}{2}\right|^2\right)\,{\rm d}t \geq \kappa_0 \left(\left\| U^0_h\right\|_1^2 +\left\| U^1_h\right\|_{M}^2\right)\qquad \left( U^0_h,\, U^1_h \in\mathbb{C}^{N}\right),
\end{equation} where $U_h=\left[u_{j}\right]_{1\leq j\leq N}$ is the solution of  the discrete homogeneous system 
\begin{equation}\label{eq.generalhom}
\left\{\begin{array}{ll}
M_h U_h''(t)+K_h U_h(t)+L_h U_h(t)=0 \qquad (t\in (0,T))\\ 
U_h(0)=U_h^0,\,\, U'_h(0)=U_h^1.
\end{array}
\right.
\end{equation}
\end{theorem}

Notice that \eqref{eq:obsineq} is a discrete version of the continuous inequality \eqref{eq_observab}. In the left hand side of  \eqref{eq:obsineq} there are two terms. The first one represents a discrete version of the continuous normal derivative $u_x(t,0)$ from \eqref{eq_observab}. The second one, coming from the non-diagonal mass matrix $M_h$, approximates $u'(t,0)=0$. As already remarked in \cite[Theorem 4.4]{CM} for the case $a\equiv 0$, both terms are necessary in order to ensure the uniformity in $h$ of the constant $\kappa_0$. On the other hand, Theorem \ref{te.obsineg} ensures the existence of a uniform time $T_0\geq 2$. However, the optimal value of $T_0$ is an open problem. 

As a consequence of Theorem \ref{te.obsineg}, we show that the mixed finite elements scheme considered in this paper can be used to approximate numerically the solution $f$ of the inverse source problem  \eqref{eq:inv_cont0}.  Indeed, a convergent sequence of approximations for $f$ will be obtained in Theorem \ref{te:coninv} below by solving a suitable discrete  optimization problem. 

Firstly, let us note that the solution $w$ of \eqref{eq:inv_cont0} verifies the relation $w'=v'+u$, where $v$ and $u$ are solutions of the following two systems
\begin{equation}\label{eq:inv_cont01}
\left\{ \begin{array}{ll} v''(t,x)-v_{xx}(t,x)+a(x)v(t,x)= \lambda(t)f(x)&
t\in(0,T),\,\, x\in(0,1)\\ v(t,0)= v(t,1)=0\,\, & t\in(0,T)\\
v(0,x)=0,\,\, v'(0,x)=0& x\in(0,1),\end{array}\right.
\end{equation}
and
\begin{equation}\label{eq:inv_cont02}
\left\{ \begin{array}{ll} u''(t,x)-u_{xx}(t,x)+a(x)u(t,x)= 0&
t\in(0,T),\,\, x\in(0,1)\\ u(t,0)= u(t,1)=0\,\, & t\in(0,T)\\
u(0,x)=w^1(x),\,\, u'(0,x)=w^0_{xx}(x)-a(x)w^0(x)& x\in(0,1),\end{array}\right.
\end{equation}
respectively. Under the hypothesis $a\in L^\infty(0,1)$, $\lambda\in H^1(0,T)$ and $\begin{pmatrix}w^0\\w^1\end{pmatrix}\in H^2 (0,1)\cap H^1_0(0,1)\times H^1_0(0,1)$, according to the classical regularity results for the wave  equation, we have that $u\in  C\left([0,T];H^1_0(0,1)\right)\cap C^1\left([0,T];L^2(0,1)\right)$ and $v\in C^2\left([0,T];L^2(0,1)\right)\cap C^1\left([0,T];H^1_0(0,1)\right)\cap  C\left([0,T];H^2(0,1)\cap H^1_0(0,1)\right)$. On one hand, this shows that $w'$ is a weak solution of a wave equation and well-known `hidden regularity' results ensure that the observation term $w_x'(\,\cdot\,,0)$ makes sense in $L^2(0,T)$ (see \cite{JLL88}).  On the other hand, this decomposition allows us to work, in the discrete setting, with initial data belonging to the energy space $H^1_0(0,1)\times L^2(0,1)$, which is natural for the mixed finite elements method.

To discretize  \eqref{eq:inv_cont01} and \eqref{eq:inv_cont02} we introduce the vectors $V_h=[v_j]_{1\leq j\leq N}$ and  $U_h=[u_j]_{1\leq j\leq N}$ which verify
\begin{equation}\label{eq.us00}
\left\{
\begin{array}{l}
M_{h}V_h''(t)+K_h V_h(t)+L_hV_h(t)=\lambda(t) M_h F_h,\\
V_h(0)=V'_h(0)=0,
\end{array}
\right.        
\end{equation}
and 
\begin{equation}\label{eq.us000}
\left\{
\begin{array}{l}
M_{h}U_h''(t)+K_h U_h(t)+L_hU_h(t)=0,\\
U_h(0)=U^0_h,\,\, U'_h(0)=U_h^1,
\end{array}
\right.        
\end{equation}
respectively. In \eqref{eq.us00} we have considered an approximation of the source term $f$ given by  $M_hF_h$, where the vector  $F_h=[f_j]_{1\leq j\leq N}\in \mathbb{C}^N$ is defined by the relation \begin{equation}\label{eq:fhdef}
M_hF_h=\left[ \frac{1}{2}\int_{x_{j-1}}^{x_{j+1}}f(x)\,{\rm d}x\right]_{1\leq j\leq N}.
\end{equation}
Moreover, the discrete initial data $U_h^0=[u_j^0]_{1\leq j\leq N}$ and $U_h^1=[u_j^1]_{1\leq j\leq N}$ in \eqref{eq.us000} are given by
\begin{equation}\label{eq:datosdisc}
U_h^0=\left[ w^1(x_j)\right]_{1\leq j\leq N},\quad M_h U_h^1=\left[ \frac{1}{2}\int_{x_{j-1}}^{x_{j+1}}\left(w^0_{xx}(x)-a(x)w^0(x)\right)\,{\rm d}x\right]_{1\leq j\leq N}.
\end{equation}
The discrete approximations \eqref{eq:fhdef} and \eqref{eq:datosdisc} have been chosen to ensure the convergence to the nonhomogeneous term and to the continuous  initial data, respectively. More precisely, we have that 
\begin{align}\label{eq:convnonhomterm}
&\sum_{j=1}^N f_j\psi_j \to f \mbox{ in }L^2(0,1) \mbox{ as }h\to 0,\\
&\sum_{j=1}^N u^0_{j}\varphi_j \to w^1 \mbox{ in }H_0^1(0,1)\quad  \mbox{ and }\quad \sum_{j=1}^N u^1_{j}\psi_j \to w^0_{xx}-aw^0 \mbox{ in }L^2(0,1) \mbox{ as }h\to 0.\label{eq:convdatosdisc}
\end{align}
In this context we are able to prove that $V_h'(t)+U_h(t)$ provides a convergent approximation for the time derivative $w'(t)$ of the solution to equation \eqref{eq:inv_cont0}. Full details concerning all these convergence results will be given in the final Appendix.

Secondly, for $T>0$ we introduce 
\begin{equation}\label{eq:obsdis00}
Y_h(t)=\begin{pmatrix}\smallskip\frac{1}{h}\left(v_{1}'(t)+u_{1}(t)\right) \\  \frac{1}{2}\left(v_{1}''(t)+u_{1}'(t)\right)\end{pmatrix}\qquad (t\in(0,T)),
\end{equation}
which stands as an approximation for the boundary observation \begin{equation}\label{eq:obscon00}y(t)=\left(\begin{matrix}w_x'(t,0)\\0\end{matrix}\right)\qquad (t\in(0,T)),\end{equation} of the continuous inverse problem \eqref{eq:inv_cont0}.  We remark that we have added in the  continuous observation term $y$ a second null component to agree with the discrete observation $Y_h$. As indicated by \eqref{eq:obsineq}, we need the two components of $Y_h$ in order to observe uniformly the solution. 

Now we study the inverse source problem for \eqref{eq.us00} which consists in recovering the discrete unknown  source vector $F_h$ (which is independent of time) from the observation term $y$, assuming that $\left(\begin{matrix}w^0\\w^1\end{matrix}\right)$ and $\lambda$ are given.  The stability of this discrete inverse problem is detailed in the following result inspired in \cite[Theorem 4.3]{ASTT}.
\begin{theorem}\label{te:estab0}
Let $T_0>0$ be given by Theorem \ref{te.obsineg}, $T> T_0$ and $\lambda \in H^1(0,T)$ with  $\lambda(0)\neq 0$. Then, there exists a positive constant $\kappa$, independent of $h$, such that 
\begin{equation} \label{eq:estab0}
\| \widehat{F}_h-\widetilde{F}_h\|_M \leq \kappa\, \| \widehat{Y}_h-\widetilde{Y}_h \|_{[L^2(0,T)]^2},
\end{equation}for any $\widehat{F}_h,\, \widetilde{F}_h\in\mathbb{C}^N$ with the corresponding solutions $\widehat{V}_h,\, \widetilde{V}_h$ of \eqref{eq.us00} and the associated observations  $\widehat{Y}_h,\, \widetilde{Y}_h$ given by  \eqref{eq:obsdis00}. Note that, in spite of the fact that $\widehat{Y}_h$ and $\widetilde{Y}_h$ depend on the solutions of  \eqref{eq.us000},  the difference $\widehat{Y}_h-\widetilde{Y}_h$ is independent of it. 
\end{theorem}
We emphasize that in estimate \eqref{eq:estab0} the constant $\kappa$ is independent of $h$, which is a direct consequence of the uniform observability inequality \eqref{eq:obsineq} for the mixed finite element scheme.

The uniform stability result given by Theorem \ref{te:estab0} allows us to use the usual least squared approach to obtain a convergent sequence of approximations of the source term $f$.  Let us describe how these approximations can be obtained and to state our convergence result. Suppose that a boundary observation $y\in \left[L^2(0,T)\right]^2$ is known, corresponding to a solution $w$ of \eqref{eq:inv_cont0} with initial data $\begin{pmatrix}w^0\\w^1\end{pmatrix} \in H^2(0,1)\cap H_0^1(0,1)\times H_0^1(0,1)$ and unknown source $f\in L^2(0,1)$. For each $h>0$, we define the functional ${\mathcal J}_h:\mathbb{C}^N\to \mathbb{R}$, 
\begin{equation}\label{eq:jn}
{\mathcal J}_h (F_h)=\frac{1}{2}\left\| Y_h- y\right\|^2_{[L^2(0,T)]^2},
\end{equation}
where $Y_h$ and $y$ are given by  \eqref{eq:obsdis00} and  \eqref{eq:obscon00}, respectively.  Now, we have all the ingredients needed to state our main convergence result.
\begin{theorem}\label{te:coninv} The functional ${\mathcal J}_h$ has a unique minimizer $\widehat{F}_h=\left[\widehat{f}_j\right]_{1\leq j\leq N}\in\mathbb{C}^N$. Moreover, if we further assume that $a\in C[0,1]$, we denote \begin{equation}\label{eq:fgoro}\widehat{f}_h=\sum_{j=1}^N \widehat{f}_{j}\psi_j \in L^2(0,1),\end{equation} and we take $T> T_0$, where $T_0>0$ is given by Theorem \ref{te.obsineg}, then the family $(\widehat{f}_h)_{h>0}$ converges to $f$ in $L^2(0,1)$ as $h$ tends to zero.
\end{theorem}

Note that  the discrete minimizer $\widehat{F}_h$ can be easily obtained by a classical minimization method applied to the functional ${\mathcal J}_h$. Therefore, the above procedure represents an efficient and simple algorithm to recover the source term $f$ in the inverse problem \eqref{eq:inv_cont0}. 

We remark that the hypothesis on the potential function $a\in C[0,1]$ is needed to show our convergence results in the Appendix and it is natural in view of the discretization by points method used  in \eqref{eq:trap} to derive our numerical scheme. However, in the continuous case, the stability result \eqref{eq:stabcont} can be proved under the weaker hypothesis $a\in L^\infty(0,1)$.

\section{Spectral analysis}

The aim of this section is to make an analysis of the spectral properties of the matrix operator corresponding to \eqref{eq.general} and to provide the Fourier expansion of the solutions of \eqref{eq.general}.  This allows us to reduce the proof of the uniform observability given in Theorem \ref{te.obsineg} to show a new result in nonharmonic Fourier series stated in Proposition \ref{prop:10} below. 

First of all we analyze the spectral properties of the matrix $M_h^{-1}(K_h+L_h)$.

\begin{lemma}\label{lemma.spec1} The matrix $M_h^{-1}(K_h+L_h)$ has  simple positive real
eigenvalues $(\mu_h^n)_{1\leq n\leq N}$ and the corresponding
eigenfunctions $(\Psi_h^n)_{1\leq n\leq N}\subset \mathbb{R}^N$
form an orthonormal basis of $\mathbb{C}^N$ with respect to the
inner product $\langle\,\cdot \,,\,\cdot\,\rangle_M$. The components of the eigenfunctions will be denoted $\Psi_h^n=\left[\phi_j^n\right]_{1\leq j\leq N}$. \end{lemma}

\begin{proof} Since the matrix $M_h$ is symmetric and positive defined, there exists a lower triangular matrix $B_h$ with
positive elements on the principal diagonal such that
$M_h=B_hB^T_h$ (Cholesky factorization). Since
$$M_h^{-1}(K_h+L_h)=(B_h^T)^{-1}\left[(B_h)^{-1}(K_h+L_h)(B_h^T)^{-1}\right]
B_h^T,$$ we deduce that the matrices $M_h^{-1}(K_h+L_h)$ and
$(B_h)^{-1}(K_h+L_h)(B_h^T)^{-1}$ have the same spectrum.  But the matrix $(B_h)^{-1}(K_h+L_h)(B_h^T)^{-1}$ is
symmetric. Therefore, all its eigenvalues
are real and there exists an orthonormal basis with respect to the canonical inner product in $\mathbb{C}^N$ formed by eigenvectors of this matrix. Let us denote by $(\mu_h^n)_{1\leq n\leq N}$ the increasing eigenvalues and $(\widehat{\Psi}_h^n)_{1\leq n\leq N}$ the normalized corresponding eigenvectors of $(B_h)^{-1}(K_h+L_h)(B_h^T)^{-1}$. 

Remark that, if $\widehat{\Psi}_h^n$ is an eigenvector of $(B_h)^{-1}(K_h+L_h)(B_h^T)^{-1}$,  then 
$\Psi_h^n:=(B_h^T)^{-1}\widehat{\Psi}_h^n$ is an eigenvector of $M_h^{-1}(K_h+L_h)$
corresponding to the same
eigenvalue. Let $\Psi_h^i$ and $\Psi_h^j$ be two eigenfunctions of the
matrix $M_h^{-1}(K_h+L_h)$. We have
that
$$\langle\Psi_h^i,\Psi_h^j\rangle_M= \langle M_h\Psi_h^i,\Psi_h^j\rangle =\left\langle
B_h\widehat{\Psi}_h^i,(B_h^T)^{-1}\widehat{\Psi}_h^j\right\rangle=\left\langle
\widehat{\Psi}_h^i,\widehat{\Psi}_h^j\right\rangle=\delta_{ij},$$
from which we deduce the desired orthogonality properties of $\Psi_h^i$ and
$\Psi_h^j$. 

To show that the eigenvalues are positive, we remark that 
$$
\mu_h^1=\left\langle
M_h^{-1}(K_h+L_h) \Psi_h^1,\Psi_h^1 \right\rangle_M=\left\langle
(K_h+L_h) \Psi_h^1,\Psi_h^1 \right\rangle >0,
$$
where the last inequality follows from the fact that the matrix  $K_h+L_h$ is positive defined. 

Finally, by taking into account the particular form of the matrices $K_h+L_h$ and $M_h$, which are tridiagonal, it follows that all the eigenvalues $\mu_h^n$ are simple. Indeed if two independent eigenvectors $\Psi_h^{n,1}$ and $\Psi_h^{n,2}$ correspond to the same eigenvalue $\mu_h^n$ then there exists real numbers $\alpha_1$  and $\alpha_2$ such that $\alpha_1^2+\alpha_2^2\neq 0$ and $\Psi_h^n= \alpha_1 \Psi_h^{n,1}+\alpha_2 \Psi_h^{n,2}$ is an eigenvector with the first component equal to zero. From the fact that $(K_h+L_h)\Psi_h^n=M_m\Psi_h^n$, it follows that 
$$\left\{\begin{array}{l}
-\frac{1}{h}\phi_{2}^n=\mu_h^n\frac{h}{4}\phi_{2}^n\\
\frac{1}{h}\left(2\phi_{2}^n+ha_2\phi_{2}^n-\phi_{3}^n\right)=\mu_h^n\frac{h}{4}\left(2\phi_{2}^n+\phi_{3}^n\right)\\...............................................................\\
\frac{1}{h}\left(-\phi_{N-1}^n+2\phi_{N}^n+ha_N\phi_{N}^n\right)=\mu_h^n\frac{h}{4}\left(\phi_{N-1}^n+\phi_{N}^n\right).
\end{array}\right.
$$
By taking into account that $\mu_h^n >0$, from the above relations we deduce that $\Psi_{h}^n=0$ which represents a contradiction. 
\end{proof}

We shall need the following result concerning the eigenfunctions $\Psi_h^n$.

\begin{lemma} (Direct inequality for eigenfunctions) \label{te.obsing}
Let $\Psi_h^n=\left[\phi_{j}^n\right]_{1\leq j\leq N}\in\mathbb{R}^N$ be the
normalized eigenvector of the matrix $M_h^{-1}(K_h+L_h)$
corresponding to the real eigenvalue $\mu_h^n$. There exist two
positive constants $h_0,C_0>0$ independent of $n$ and $a$ such that
\begin{equation}\label{eq.obsingeig}
\left(\mu_h^n +\frac{1}{h^2}
\right)\left|\frac{\phi_{1}^n}{\sqrt{\mu_h^n}}\right|^2 \leq C_0\left(1+\frac{\|a\|_{L^\infty(0,1)}^2}{\mu_h^n}\right)\quad \left(h\in (0,h_0)\right).
\end{equation}
\end{lemma}

\begin{proof} We have that $ (K_h+L_h)\Psi_{h}^n=\mu_h^n M_h \Psi_{h}^n$, which is
equivalent to
\begin{equation}\label{eq.e1j}
\frac{1}{h^2}\left(-\phi^n_{j+1}+2\phi^n_{j}-\phi^n_{j-1}\right)+a_j\phi^n_{j}=
\frac{\mu_h^n}{4}\left(\phi^n_{j+1}+2\phi^n_{j}+\phi^n_{j-1}\right)\,
(1\leq j\leq N).
\end{equation}
Note that in \eqref{eq.e1j} the extra components $\phi_0$ and $\phi_{N+1}$ are required. Recall that, in  this case, $\phi_0=\phi_{N+1}=0$, by convention. 
Multiplying each equation \eqref{eq.e1j} by $\phi^n_{j}$ and adding
the relations we obtain that
\begin{equation}\label{eq.e2j}
\sum_{j=0}^N
\left|\frac{\phi^n_{j+1}-\phi^n_{j}}{h}\right|^2+\sum_{j=1}^N a_j
\left|\phi^n_{j}\right|^2=\mu_h^n\sum_{j=0}^N
\left|\frac{\phi^n_{j+1}+\phi^n_{j}}{2}\right|^2.
\end{equation}
From \eqref{eq.e2j} and using the fact that $$|\phi^n_{j}|^2\leq
\frac{1}{2}|\phi^n_{j+1}+\phi^n_{j}|^2+\frac{1}{2}|\phi^n_{j+1}-\phi^n_{j}|^2,$$
we deduce that
\begin{equation}\label{eq.e2jineg}
\left(1-\frac{h^2}{2}\|a\|_{L^\infty(0,1)}\right)\sum_{j=0}^N
\left|\frac{\phi^n_{j+1}-\phi^n_{j}}{h}\right|^2\leq
\left(\mu_h^n+2\|a\|_{L^\infty(0,1)}\right)\sum_{j=0}^N
\left|\frac{\phi^n_{j+1}+\phi^n_{j}}{2}\right|^2.
\end{equation}

On the other hand, multiplying each equation (\ref{eq.e1j}) by the
multiplier $(h^{-1}-j)(\phi^n_{j+1}-\phi^n_{j-1})$ and adding in $j=1,...,N$, it follows that
\begin{eqnarray}\label{eq.e3j}
 0&=&\frac{\mu}{4}\sum_{j=1}^N \left(\phi^n_{j+1}+\phi^n_{j-1}+2
\phi^n_{j}\right)(h^{-1}-j) (\phi^n_{j+1}-\phi^n_{j-1})+\\
\nonumber && +\frac{1}{h^2}\sum_{j=1}^N
\left(\phi^n_{j+1}+\phi^n_{j-1}-2\phi^n_{j}\right)
(h^{-1}-j)(\phi^n_{j+1}-\phi^n_{j-1})\\ \nonumber
&& - \sum_{j=1}^N a_j \phi^n_{j} \,
(h^{-1}-j)(\phi^n_{j+1}-\phi^n_{j-1}).
\end{eqnarray}
Let us estimate each term from \eqref{eq.e3j}. We have that
\begin{equation} \label{eq.33bj1}
\sum_{j=1}^N\left(\phi^n_{j+1}+\phi^n_{j-1}-2\phi^n_{j}\right)(h^{-1}-j)
\left(\phi^n_{j+1}-\phi^n_{j-1}\right)=\sum_{j=0}^N
|\phi^n_{j+1}-\phi^n_{j}|^2 -\frac{|\phi^n_{1}|^2}{h} ,
\end{equation}
\begin{align} \label{eq.33bj2} 
\sum_{j=1}^N\left[\phi^n_{j+1}+\phi^n_{j-1}+2\phi^n_{j}\right](h^{-1}-j)
\left(\phi^n_{j+1}-\phi^n_{j-1}\right)
=\sum_{j=0}^N
|\phi^n_{j+1}+\phi^n_{j}|^2 -\frac{|\phi^n_{1}|^2}{h} ,
\end{align}
and
\begin{align} \nonumber
&& \left|\sum_{j=1}^N a_j \phi^n_{j} \,
(h^{-1}-j)  (\phi^n_{j+1}-\phi^n_{j-1})\right| \leq \|a\|_{L^\infty(0,1)}
\sum_{j=1}^N|\phi^n_{j}|\left[\left|\frac{\phi^n_{j+1}-\phi^n_{j}}{h}\right|
+\left|\frac{\phi^n_{j}-\phi^n_{j-1}}{h}\right|\right]\\
\nonumber  & & \leq \|a\|_{L^\infty(0,1)}
\left[\sum_{j=1}^N|\phi^n_{j}|^2\right]^\frac{1}{2}
\left[\sum_{j=1}^N\left(\left|\frac{\phi^n_{j+1}-\phi^n_{j}}{h}\right|
+\left|\frac{\phi^n_{j}-\phi^n_{j-1}}{h}\right|\right)^2\right]^\frac{1}{2}\\
\nonumber && \leq \frac{\|a\|_{L^\infty(0,1)}^2}{2}
\sum_{j=1}^N\left|\phi^n_{j}\right|^2+ \sum_{j=0}^N
\left|\frac{\phi^n_{j+1}-\phi^n_{j}}{h}\right|^2
\\
\label{eq.33bj3} && \leq \|a\|_{L^\infty(0,1)}^2
\sum_{j=0}^N\left|\frac{\phi^n_{j+1}+\phi^n_{j}}{2}\right|^2+
\left(1+\frac{h^2}{4}\|a\|_{L^\infty(0,1)}^2\right)\sum_{j=0}^N
\left|\frac{\phi^n_{j+1}-\phi^n_{j}}{h}\right|^2.
\end{align}
From \eqref{eq.e3j}-\eqref{eq.33bj3} we obtain that
\begin{eqnarray}\nonumber
& \displaystyle \frac{1}{4h}\left(h^2 \mu^n_h
+4\right)\left|\frac{\phi^n_{1}}{h}\right|^2 &\leq \left(\mu_h^n +\|a\|_{L^\infty(0,1)}^2\right) \,
\sum_{j=0}^N\left|\frac{\phi^n_{j+1}+\phi^n_{j}}{2}\right|^2\\ \label{eq.inobs}&&
+\left(2+\frac{h^2}{4}\|a\|_{L^\infty(0,1)}^2\right)\sum_{j=0}^N
\left|\frac{\phi^n_{j+1}-\phi^n_{j}}{h}\right|^2.
\end{eqnarray}
From \eqref{eq.e2jineg} and \eqref{eq.inobs} we deduce that, for
$h$ sufficiently small, the following inequality holds
\begin{eqnarray}\label{eq.inobs2}
\frac{1}{4h}\left(h^2 \mu_h^n
+4\right)\left|\frac{\phi^n_{1}}{h}\right|^2 \leq C \left(\mu_h^n
+\|a\|_{L^\infty(0,1)}^2\right)\,
\sum_{j=0}^N\left|\frac{\phi^n_{j+1}+\phi^n_{j}}{2}\right|^2.
\end{eqnarray}
Since the eigenvector $\Psi^n_h=\left[\phi^n_{j}\right]_{1\leq j\leq
N}\in\mathbb{R}^N$ is normalized, from \eqref{eq.inobs2} we deduce
\eqref{eq.obsingeig} and the proof of the proposition is complete.
\end{proof}

Moreover, in $\mathbb{C}^{2N}$, we define the inner product
\begin{equation}\label{eq.inprod3}
\displaystyle \left\langle\begin{pmatrix}
U^1\\ U^2
\end{pmatrix} ,\begin{pmatrix}
W^1\\ W^2
\end{pmatrix}\right\rangle_{1,M}=\langle K_h
U^1,W^1\rangle +\langle L_h U^1,W^1\rangle + \langle M_h
U^2,W^2\rangle,
\end{equation}
where $U^i,\,W^i\in \mathbb{C}^{N}$ for $i\in\{1,2\}$. The
corresponding norm will be denoted by $||\,\, \cdot \,\,||_{1,M}$. 

We can write \eqref{eq.generalhom} in the following equivalent form
\begin{equation}
\left\{
\begin{array}{ll}
\begin{pmatrix}
U_{h} \\
U_{h}^{\prime }
\end{pmatrix}
 ^{\prime }(t)+{\mathcal A}_h\begin{pmatrix}
U_{h} \\
U_{h}^{\prime }
\end{pmatrix} (t)=\begin{pmatrix}
0 \\
0
\end{pmatrix} & (t\in (0,T)) \\ \\ \begin{pmatrix}
U_{h}\\
U_{h}^{\prime }
\end{pmatrix} (0)=\begin{pmatrix}
U_{h}^0 \\
U_{h}^1
\end{pmatrix},
\end{array}
\right.   \label{eq.general1}
\end{equation}
where the operator ${\mathcal A}_h$ is defined by  \begin{equation}
\label{eq.opmat}
{\mathcal A}_h=\begin{pmatrix} 0 & -I \\
M_h^{-1}(K_h +L_h) &0 \end{pmatrix}.
\end{equation}

We can pass to study the spectral properties of the matrix ${\mathcal A}_h$ defined by \eqref{eq.opmat}. We have the following result.

\begin{lemma}\label{lemma.spec2} The matrix ${\mathcal A}_h$ has a family
of imaginary eigenvalues $(i\, \lambda_h^n)_{1\leq |n|\leq
N}$ where \begin{equation}\label{eq.lamba}
\lambda_h^n=\mbox{\,sgn\,}(n)\sqrt{\mu_h^{|n|}}
\qquad (1\leq |n|\leq N).
\end{equation}
A normalized eigenvector of ${\mathcal A}_h$ corresponding to a nonzero eigenvalue $i\, \lambda_h^n$ is given by
\begin{equation}\label{eq.eifugen} \Phi_h^n =\frac{1}{\sqrt{2 \mu_h^{|n|}}}\begin{pmatrix} \Psi_h^{|n|}\\-i\, \lambda_h^n \, \Psi_h^{|n|}\end{pmatrix}.
\end{equation} 
Moreover, $(\Phi_h^n)_{1\leq |n|\leq N}$ forms an orthonormal basis of
 $\mathbb{C}^{2N}$ with respect to the inner product $\langle
\,\cdot\,,\,\cdot\,\rangle_{1,M}$ given by \eqref{eq.inprod3}.\end{lemma}

\begin{proof}  The fact that the eigenvalues and the corresponding eigenvectors of ${\mathcal A}_h$ are given by \eqref{eq.lamba} and \eqref{eq.eifugen}, respectively, is immediate. Notice that the property $\mu_h^n>0$ implies that $\lambda_h^n \in \mathbb{R}$. The fact that
$(\Phi_h^n)_{1\leq |n|\leq N}$ form an orthonormal basis of
 $\mathbb{C}^{2N}$ is a consequence of the orthogonality of the eigenfunctions $(\varphi_h^n)_{1\leq n\leq N}$.
\end{proof}

Using the above basis property and the Fourier expansion of the initial data $\begin{pmatrix}U^0_h\\U^1_h\end{pmatrix}$, any solution of (\ref{eq.general1}) can be written as follows
\begin{equation}\label{eq.solfour}
\begin{pmatrix}
U_{h} \\
U_{h}'
\end{pmatrix}(t)=\sum_{1\leq |n|\leq N} \left\langle \begin{pmatrix}U^0_h\\U^1_h\end{pmatrix}, \Phi_h^n\right\rangle_{1,M} e^{-i\, \lambda^n_h t} \Phi^n_h.\end{equation}
By taking into account \eqref{eq.solfour}, Theorem \ref{te.obsineg} is equivalent to the following nonharmonic Fourier result.
\begin{proposition} \label{prop:10}
Under the hypotheses of Theorem \ref{te.obsineg}, there exist two constants $T_0,\, \kappa_0>0$, independent of $h$, such that the following inequality holds  
\begin{equation} \label{eq:prop10}
\int_0^{T_0}\left( \left| \sum_{1\leq |n|\leq N} b^n  \frac{\phi^{n}_{1}}{h\sqrt{\mu_h^{|n|}}}e^{-i\lambda^n_h
t} \right|^2 + \left| \sum_{1\leq |n|\leq N} b^n  \mbox{\,sgn\,}(n) \frac{\phi^{n}_{1}}{2} e^{-i\lambda^n_h
t} \right|^2\right)\mbox{\,d}t  \geq \kappa_0 \sum_{1\leq |n|\leq N} |b^n|^2 ,
\end{equation} 
for any sequence of complex numbers $(b^n)_{1\leq |n|\leq N}$.
\end{proposition}

\begin{remark} One of the most remarkable aspects of \eqref{eq:prop10} is the independence of the constant $\kappa_0$ on the discretization parameter $h$. This property does not hold in the case of centered finite differences or finite elements schemes. Indeed, as shown in \cite{IZ}, the corresponding observability constant explodes as $h$ tends to zero. This is mainly due to the fact that there exist spurious numerical high eigenvalues which are not uniformly separated and, consequently, are not well observable from the boundary. In the case of mixed finite elements, although the high frequencies are still poorly approximated, the uniform gap property of the continuous case is maintained.

In the case of inverse problems, an observation which is not uniform with respect to $h$ does not allow to reconstruct correctly the unknown terms. Consequently, an uniform observability inequality like \eqref{eq:prop10} represents an important step in order to solve numerically an inverse problem and illustrates the utility of the mixed finite elements in this type of endeavors. 
\end{remark}

\section{Proof of the uniform observability result}\label{sec:4}

In this section we prove Proposition \ref{prop:10}. As we mentioned above this is equivalent to Theorem \ref{te.obsineg}. The proof is easily obtained by showing that there exists $h_0>0$ such that for any $h\in(0,h_0)$ the following two properties hold:
\begin{itemize}
\item A uniform bound from below for the eigenvectors, i.e.
\begin{equation} \label{eq:ing1}
\left| \frac{\phi^{n}_{1}}{h\sqrt{\mu_h^{|n|}}} \right|^2+\left| \frac{\phi^{n}_{1}}{2} \right|^2 \geq c >0, \mbox{ for all $|n|\leq N$},
\end{equation}
Note that, apart from the first term representing the natural approximation of the normal derivative at $x=0$, it is necessary to add a second term to guarantee the uniformity of the constant $c$. This is due to the poor numerical approximation of the normal derivative of the highest eigenmodes and it was also observed in the case of the classical centered finite difference scheme (see, \cite[Lemma 1.1]{IZ}). As we have already mentioned, this second term comes naturally from the mass matrix associated to the mixed finite elements scheme.

\item The following result in nonharmonic Fourier series: there exist $T_0,C>0$, independent of $N$, such that 
\begin{equation} \label{eq:ing2}
\sum_{1\leq |n|\leq N} |b^n|^2 \leq C\int_0^{T} \left| \sum_{1\leq |n|\leq N} b^n e^{-i\lambda^n_h
t} \right|^2\,{\rm d}t, 
\end{equation}
holds for any $T>T_0$ and all $(b^n)_{1\leq |n|\leq N} \subset \mathbb{C}$.
\end{itemize}

The rest of this section is devoted to prove \eqref{eq:ing1} and \eqref{eq:ing2}. We start with (\ref{eq:ing2}). According to Ingham's theorem \cite{I}, \eqref{eq:ing2} holds for $T_0=\frac{2\pi}{\gamma_0}$, where $\gamma_0$ is the spectral gap of the eigenvalues corresponding to the operator $\mathcal{A}_h$. The fact that this gap is uniform with respect to $N$ follows from the following theorem. 

\begin{theorem} \label{te:spectralgap} Assume that \eqref{eq:pot} holds. There exist two positive constants $h_0$ and $\gamma_0$ such that for any two different eigenvalues $\lambda_h^n$ and $\lambda_h^m$
of \eqref{eq.opmat}, we have that
\begin{equation}\label{eq.gap2eig}
\left|\lambda_h^n-\lambda_h^m\right|\geq \gamma_0 \qquad (h\in(0,h_0)).
\end{equation}
\end{theorem}
\begin{proof} Given any $\tau>2$, suppose that $\lambda_h^n$ and $\lambda_h^m$ are two different eigenvalues with the property that 
\begin{equation}
\label{eq.gmic}\left|\lambda_h^n-\lambda_h^m\right|<\frac{\pi}{2\tau }.
\end{equation}
Let $\Phi^n_h$ and $\Phi^m_h$ be two unitary eigenfunctions of the operator ${\mathcal A}_h$  corresponding to the eigenvalues $\lambda_h^n$ and $\lambda_h^m$, respectively (see Lemma \ref{lemma.spec2} for the definition of the eigenfunctions and eigenvalues of the operator ${\mathcal A}_h$).  
We'll show that the following estimate holds true:
\begin{eqnarray}  \nonumber
 |b^n|^2+|b^m|^2 &\leq & C_\tau \left[ \int_0^\tau \left|  b^n e^{-i\lambda^n_h
t}\frac{\phi^{n}_{1}}{h\sqrt{\mu_h^{|n|}}}+ b^m e^{-i\lambda^m_h
t} \frac{\phi^{m}_{1}}{h\sqrt{\mu_h^{|n|}}} \right|^2 \,{\rm d}t  \right.  \\
\label{eq:prop10b}
&& \left. + \int_0^\tau \left|  b^n e^{-i\lambda^n_h
t} \mbox{\,sgn\,}(n) \frac{\phi^{n}_{1}}{2}+ b^m e^{-i\lambda^m_h
t} \mbox{\,sgn\,}(m) \frac{\phi^{m}_{1}}{2} \right|^2\,{\rm d}t \right] ,
\end{eqnarray} 
for all $b^n,b^m\in \mathbb{C}$, where $C_\tau$ is a positive constant independent of $h$, $n$ and $m$ (depending only on $\tau $ and the potential $a$). Moreover, from \eqref{eq:prop10b} we'll deduce that the following gap property holds 
\begin{equation} \label{eq:pp2}
|\lambda_h^n-\lambda_h^m|\geq \sqrt{\frac{3}{2C_\tau C_0 \tau^3}},
\end{equation}where $C_0$ and $C_\tau$ are the constants in \eqref{eq.obsingeig} and \eqref{eq:prop10b}, respectively. The conclusion of our theorem follows by taking \begin{equation}\label{eq:gamma0}\gamma_0=\min\left\{\frac{\pi}{2\tau} ,\, \sqrt{\frac{3}{2C_\tau C_0 \tau^3}}\right\}.\end{equation} Properties  \eqref{eq:prop10b} and \eqref{eq:pp2} are proved in Propositions \ref{prop:11} and \ref{prop:12} below. 
\end{proof}

\begin{proposition} \label{prop:11}
Assume that the hypotheses of Theorem  \ref{te:spectralgap} hold and consider $\tau>2$. If $\lambda^n_h$ and $\lambda_h^m$ are two different eigenvalues of \eqref{eq.opmat} verifying \eqref{eq.gmic}, then  estimate (\ref{eq:prop10b}) holds.
\end{proposition}

\begin{remark}\label{rem:11}
 Notice that estimate (\ref{eq:prop10b}) is a version of  \eqref{eq:prop10}  when we consider a solution consisting on only two terms.  A natural idea in order to prove the discrete observability inequality  \eqref{eq:prop10} is to mimic the strategy, based on lateral energy estimates, used to show  \eqref{eq_observab} for the continuous system \eqref{eq.wave} in \cite{Z}. Let us remind its main steps:
\begin{enumerate}
\item Define the function
\begin{equation} \label{eq_cce}
E(t,x):=\frac12 (|u_t(t,x)|^2+|u_x(t,x)|^2+\|a\|_{L^\infty(0,1)}|u(t,x)|^2),
\end{equation}
where $u$ is the solution of \eqref{eq.wave}, and consider for $T>2$ and $1<\beta<T/2$ 
\begin{equation} \label{eq_cce2}
F(x)=\int_{\beta x}^{T-\beta x} E(s,x) \,{\rm d}s.
\end{equation}
Prove that there exists a constant $c_a>0$, depending only on the norm $\|a\|_{L^\infty(0,1)}$, such that  
\begin{equation} \label{eq:gr1}
F'(x)\leq c_a F(x)
\end{equation}
\item Use Gronwall's inequality to obtain a new constant $c'_a=\exp(c_a)$ such that
\begin{equation} \label{eq:gr2}
F(x)\leq c'_a F(0).
\end{equation}
Then, we apply the conservation of the energy, i.e. $\mathcal{E}(t)=\int_0^1\tilde E(t,x)\; \,{\rm d}x=\mathcal{E}(0)$, where 
$$
\tilde E(t,x):=\frac12 (|u_t(t,x)|^2+|u_x(t,x)^2+a(x)|u(t,x)|^2) \leq E(t,x),
$$
to obtain inequality \eqref{eq_observab} as follows,
\begin{eqnarray} \nonumber 
(T-2\beta) \mathcal{E}(0) &=&\int_{\beta}^{T-\beta }\int_{0}^{1}\tilde E(t,x)\,{\rm d}x\,{\rm d}t\leq
\int_{0}^{1}\int_{\beta x}^{T -\beta x}E(s,x)\,{\rm d}s\,{\rm d}x \\ \label{eq:gr3}
&=&\int_{0}^{1}F(x)\,{\rm d}x\leq c_a' F(0) =c'_a\int_{0}^{T }E(s,0)\,{\rm d}s = \frac{c'_a}{2}\int_{0}^{T }|u_{x}(s,0)|^2\,{\rm d}s.
\end{eqnarray}
\end{enumerate}
Unfortunately, when trying the same strategy for the discrete system some extra terms appear in the first step of this process which do not allow to recover the natural discrete version of the inequality (\ref{eq:gr1}) directly. However, we can deal with these extra terms when  particular solutions with only two different modes are considered, as soon as the associated eigenvalues satisfy the hypotheses \eqref{eq.gmic}. We give the details of the proof below. 
\end{remark}

{\bf Proof of Proposition \ref{prop:11}.}
We divide the proof in different steps according to the strategy outlined in Remark \ref{rem:11} above. 

\bigskip

{\bf Step 1:} Let $U_h=\left[u_j\right]_{1\leq
j\leq N}$ be a particular solutions of equation \eqref{eq.general1}.  We first introduce the discrete versions of the function $E(t,x)$ in \eqref{eq_cce} by 
\begin{equation}\label{eq:ej}
E_{j}(s)=\left| \frac{u_{j+1}(s)-u_{j}(s)}{h}\right| ^{2}+\left|
\frac{u_{j+1}^{\prime }(s)+u_{j}^{\prime }(s)}{2}\right| ^{2}+a_{M}\left| \frac{u_{j+1}(s)+u_{j}(s)}{2}\right| ^{2},
\end{equation}
where $0\leq j\leq N$ and $a_M=\max_{j=0,...,N+1} |a_j|$. Let $1<\beta<\tau/2 $ and consider the discrete version of $F(x)$ in \eqref{eq_cce2}:
\[
F_{j}=\frac{1}{2}\int_{\beta x_{j}}^{\tau -\beta x_{j}}E_{j}(s)\,{\rm d}s.
\]
In this step we prove the following discrete version of \eqref{eq:gr1}: 
\begin{eqnarray} \nonumber
&& \frac{F_{j}-F_{j-1}}{h} \leq c_a\frac{F_{j}+F_{j-1}}{2}+ \frac{c_a}{4}\left[
\int_{\beta x_{j-1}}^{\beta x_{j}} E_j(s)\,{\rm d}s + \int_{\tau-\beta
x_{j}}^{\tau-\beta x_{j-1}} E_j(s)\,{\rm d}s \right] \\ \nonumber
&&+\frac{1}{2}\left[ \frac{\frac{E_{j}+E_{j-1}}{2}(\tau -\beta x_{j})+\frac{%
E_{j}+E_{j-1}}{2}(\tau -\beta x_{j-1})}{2}-\frac{1}{h}\int_{\tau
-\beta
x_{j}}^{\tau -\beta x_{j-1}}\frac{E_{j}+E_{j-1}}{2}(s)\,{\rm d}s\right] \\ \label{eq:stp1}
&&+\frac{1}{2}\left[ \frac{\frac{E_{j}+E_{j-1}}{2%
}(\beta x_{j})+\frac{E_{j}+E_{j-1}}{2}(\beta x_{j-1})}{2}-\frac{1}{h}%
\int_{\beta x_{j-1}}^{\beta x_{j}}\frac{E_{j}+E_{j-1}}{2}(s)\,{\rm d}s\right],
\end{eqnarray}%
where $c_a=\frac{3}{2}\sqrt{a_M}\left(1+\frac{h \sqrt{a_M}}{3}\right).$

Taking into account that $\beta x_{j-1}<\beta x_j <\tau-\beta x_j <\tau - \beta x_{j-1}$, we have
\begin{eqnarray} \nonumber
\frac{F_{j}-F_{j-1}}{h} &=&\frac{1}{4h}\int_{\beta x_{j}}^{\tau -\beta x_{j}}\left(
E_{j}(s)-E_{j-1}(s)\right) \,{\rm d}s+\frac{1}{4h}\int_{\beta x_{j-1}}^{\tau -\beta
x_{j-1}}\left( E_{j}(s)-E_{j-1}(s)\right) \,{\rm d}s \\ \label{eq_le_0}
&&-\frac{1}{4h}\int_{\beta x_{j-1}}^{\beta x_{j}}\left(
E_{j-1}(s)+E_{j}(s)\right) \,{\rm d}s-\frac{1}{4h}\int_{\tau -\beta x_{j}}^{\tau
-\beta x_{j-1}}\left( E_{j-1}(s)+E_{j}(s)\right) \,{\rm d}s.
\end{eqnarray}%
The last two terms here are already in the right hand side in \eqref{eq:stp1}. Thus, we only have to estimate the first two terms in the right hand side. Using the elementary relation $|z_1|^2-|z_2|^2=\Re ((z_1+z_2)(\overline{z_1-z_2}))$ and integrating by parts we obtain
\begin{eqnarray} \nonumber
&&\frac{1}{2h}\int_{\beta x_{j}}^{\tau -\beta x_{j}}\left(
E_{j}-E_{j-1}\right) \,{\rm d}s =\Re \int_{\beta x_{j}}^{\tau -\beta x_{j}}\left( \frac{%
u_{j+1}-2u_{j}+u_{j-1}}{h^{2}}-\frac{u_{j+1}^{\prime \prime }+2u_{j}^{\prime
\prime }+u_{j-1}^{\prime \prime }}{4} \right. \\ \nonumber
&& \left. +a_{M}\frac{u_{j+1}+2u_{j}+u_{j-1}}{4}\right) \frac{\overline{u_{j+1}-u_{j-1}}%
}{2h}\,{\rm d}s +\Re \left[ \frac{u_{j+1}^{\prime }+2u_{j}^{\prime }+u_{j-1}^{\prime }}{4}%
\frac{\overline{u_{j+1}-u_{j-1}}}{2h}\right] _{\beta x_{j}}^{\tau -\beta x_{j}} \\ \nonumber
&&=  \Re \int_{\beta x_{j}}^{\tau -\beta x_{j}}\left( a_j u_j +a_{M}\frac{u_{j+1}+2u_{j}+u_{j-1}}{4}\right) \frac{\overline{u_{j+1}-u_{j-1}}%
}{2h}\,{\rm d}s \\ \label{eq_le_1}
&&+\Re \left[ \frac{u_{j+1}^{\prime }+2u_{j}^{\prime }+u_{j-1}^{\prime }}{4}%
\frac{\overline{u_{j+1}-u_{j-1}}}{2h}\right] _{\beta x_{j}}^{\tau -\beta x_{j}}.
\end{eqnarray}%
We estimate separately the two terms in the right hand side. Concerning the first term, we have on one hand
\begin{eqnarray} \nonumber
&& \displaystyle a_{M}\Re \left(\frac{u_{j+1}+2u_{j}+u_{j-1}}{4}\frac{\overline{u_{j+1}-u_{j-1}}%
}{2h} \right) \\ \nonumber
&& \quad  \leq \frac{a_M}{4}  \left| \frac{u_{j-1}+u_j}{2} +
\frac{u_{j}+u_{j+1}}{2} \right|\left|  \frac{\overline{u_{j}-u_{j-1}}}{h} +
\frac{\overline{u_{j+1}-u_{j}}}{h}\right|\\ \nonumber  
&& \quad  \leq
\displaystyle \frac{\sqrt{a_M}}{4} \left( a_M \left| \frac{u_{j-1}+u_j}{2}
\right|^2 + \left|\frac{u_{j}-u_{j-1}}{h} \right|^2 + a_M \left|
\frac{u_{j}+u_{j+1}}{2} \right|^2 + \left|\frac{u_{j+1}-u_{j}}{h}
\right|^2  \right) \\ \label{eq_le_2}  && \quad \leq \displaystyle
\frac{\sqrt{a_M}}{4} (E_{j-1}+E_j),
\end{eqnarray}
where we have used Young's inequality. On the other hand, we have
\begin{eqnarray} \nonumber
&& a_j \Re \left( u_j \frac{\overline{u_{j+1}-u_{j-1}}}{2h} \right) \leq  \frac{a_M}{4} \left| \frac{u_{j-1}+u_j}{2} + \frac{u_{j}+u_{j+1}}{2} + \frac{u_{j}-u_{j-1}}{2} + \frac{u_{j}-u_{j+1}}{2}\right| \\ \nonumber
&&\quad \times \left(  \frac{|u_{j}-u_{j-1}|}{h} + \frac{|u_{j+1}-u_{j}|}{h}\right) \\ \nonumber
&& \leq  \frac{\sqrt{a_M}}{2} \left( a_M \frac{|u_{j-1}+u_j|^2}{4} + \frac{|u_{j}-u_{j-1}|^2}{h^2}  + a_M  \frac{|u_{j}+u_{j+1}|^2}{4} + \frac{|u_{j+1}-u_{j}|^2}{h^2} \right) \\ \nonumber
&& \quad+  \frac{ha_M}{4} \left( \frac{|u_{j}-u_{j-1}|^2}{h^2} +
\frac{|u_{j+1}-u_{j}|^2}{h^2}  \right) \\ \label{eq_le_3}
&& \leq \frac{\sqrt{a_M}}{2}(1+\frac{h \sqrt{a_M}}{2}) (E_{j-1}+E_j).
\end{eqnarray}
The second term in the right hand side in (\ref{eq_le_1}) is estimated from the inequality,
\begin{eqnarray} \nonumber
&&\Re \left(\frac{u_{j+1}^{\prime }+2u_{j}^{\prime }+u_{j-1}^{\prime }}{4}\frac{%
\overline{u_{j+1}-u_{j-1}}}{2h} \right) \\ \nonumber
&\leq &\frac{1}{2}\left| \frac{u_{j+1}^{\prime }+2u_{j}^{\prime
}+u_{j-1}^{\prime }}{4}\right| ^{2}+\frac{1}{2}\left| \frac{u_{j+1}-u_{j-1}}{%
2h}\right| ^{2} \\ \nonumber
&\leq &\left| \frac{u_{j+1}^{\prime }+u_{j}^{\prime }}{4}\right| ^{2}+\left|
\frac{u_{j}^{\prime }+u_{j-1}^{\prime }}{4}\right| ^{2}+\left| \frac{%
u_{j+1}-u_{j}}{2h}\right| ^{2}+\left| \frac{u_{j}-u_{j-1}}{2h}\right| ^{2} \\  \label{eq_le_4}
&\leq &\frac{E_{j}+E_{j-1}}{4}
\end{eqnarray}%
Combining (\ref{eq_le_2})-(\ref{eq_le_4}) we estimate the right hand side in (\ref{eq_le_1}) and we obtain
\begin{eqnarray*}
&& \frac{1}{2h}\int_{\beta x_{j}}^{\tau -\beta x_{j}}\left(
E_{j}-E_{j-1}\right) \,{\rm d}s \leq \frac{3}{4}\sqrt{a_M}(1+\frac{h \sqrt{a_M}}{3}) \int_{\beta x_{j}}^{\tau -\beta x_{j}} \frac{E_{j-1}+E_j}{2}\,{\rm d}s \\
&& \quad + \frac{E_{j}+E_{j-1}}{4} (\beta x_{j}) + \frac{E_{j}+E_{j-1}}{4} (\tau -\beta x_{j}).
\end{eqnarray*}
An analogous inequality is obtained if we integrate $E_j-E_{j-1}$ in the interval $(\beta x_{j-1},\tau-\beta x_{j-1})$. Therefore, formula (\ref{eq_le_0}) is estimated by,
\begin{align*} 
\frac{F_{j}-F_{j-1}}{h} &\leq \frac{1}{2}\left[ \frac{\frac{E_{j}+E_{j-1}}{2%
}(\beta x_{j})+\frac{E_{j}+E_{j-1}}{2}(\beta x_{j-1})}{2}-\frac{1}{h}%
\int_{\beta x_{j-1}}^{\beta x_{j}}\frac{E_{j}+E_{j-1}}{2}\,{\rm d}s\right] \\ 
&+\frac{1}{2}\left[ \frac{\frac{E_{j}+E_{j-1}}{2}(\tau -\beta x_{j})+\frac{%
E_{j}+E_{j-1}}{2}(\tau -\beta x_{j-1})}{2}-\frac{1}{h}\int_{\tau -\beta
x_{j}}^{\tau -\beta x_{j-1}}\frac{E_{j}+E_{j-1}}{2}\,{\rm d}s\right] \\
&+\frac{3}{8}\sqrt{a_M}(1+\frac{h \sqrt{a_M}}{3}) \left( \int_{\beta x_{j}}^{\tau -\beta x_{j}} \frac{E_{j-1}+E_j}{2}\,{\rm d}s + \int_{\beta x_{j-1}}^{\tau -\beta x_{j-1}} \frac{E_{j-1}+E_j}{2}\,{\rm d}s \right).
\end{align*}%
Finally, the last term can be written in terms of $F_j$ as follows
$$
\int_{\beta x_{j}}^{\tau -\beta x_{j}} \frac{E_{j-1}+E_j}{2}\,{\rm d}s +
\int_{\beta x_{j-1}}^{\tau -\beta x_{j-1}} \frac{E_{j-1}+E_j}{2}\,{\rm d}s
\leq 4 \frac{F_j+F_{j-1}}{2} + \int_{\beta x_{j-1}}^{\beta x_j}
E_j \,{\rm d}s + \int_{\tau \beta x_{j}}^{\tau -\beta x_{j-1}} E_j \,{\rm d}s.
$$
This concludes the proof of \eqref{eq:stp1} and the Step 1.

\bigskip

{\bf Step 2.} We prove that there exists $h_0>0$ such that, for any $h\in(0,h_0)$,  estimate \eqref{eq:stp1} implies that   
\begin{eqnarray} \label{eq:step2}
\frac{F_{j}-F_{j-1}}{h} &\leq &c_a\frac{F_{j}+F_{j-1}}{2}, \qquad
1\leq j \leq N+1,
\end{eqnarray}
when we consider particular solutions $U_h=(u_j)_{1\leq
j\leq N}$ of equation \eqref{eq.general1} of the form
\begin{equation}\label{eq.p}
\begin{pmatrix}U_h\\U_h'\end{pmatrix}(t)=b^n
e^{-i\lambda_h^n t}\Phi^n_h + b^m e^{-i\lambda_h^m
t}\Phi^m_h.
\end{equation}
The proof of \eqref{eq:step2} relays on estimating the last two terms in the right hand side of \eqref{eq:stp1}. This is obtained by using the following technical lemma which gives a nonstandard estimate of the error in the trapezoidal quadrature formula.

\begin{lemma} \label{le_c1} Let $r>0$, $t\geq 0$ and $\nu_1$, $\nu_2$ be two different real numbers such that,
\begin{equation}\label{eq:condgap}
|\nu_2-\nu_1| \leq \frac{\pi}{2t+r}.
\end{equation}
Then, the following estimate holds
\begin{equation} \label{desig}
\frac{f(t)+f(t+r)}{2} \leq \frac{1}{r}\int_{t}^{t+ r} f(s)\,{\rm d}s ,
\end{equation}
for any function $f(t)$ of the form
$$
f(t)=\left| b_1 e^{i\nu_1 t}+b_2 e^{i\nu_2t}  \right|^2,
$$
with $b_1,b_2\in\mathbb{C}$.
\end{lemma}
\begin{proof} Note that \eqref{desig} holds (with equality) if $b_1$ or $b_2$ is zero. Therefore, it is sufficient to show  \eqref{desig} for functions of the form $f(t)=\left|b e^{i\zeta t}+1  \right|^2$. In this case we have that
$$
f(t)=b^2+1+2b\cos(\zeta t),
$$
and \eqref{desig} is equivalent to 
\begin{equation}\label{eq:desig1}
 \frac{2}{\zeta r} \cos\left(\frac{\zeta(2t+r)}{2}\right)\left[ \sin\left(\frac{\zeta r}{2}\right)- \frac{\zeta r}{2}\cos\left(\frac{\zeta r}{2}\right)\right]\geq 0.
\end{equation}
It follows that, under the hypothesis \eqref{eq:condgap}, estimate \eqref{desig} holds and the proof of the lemma is complete.
\end{proof}

We now continue with the proof of estimate \eqref{eq:step2}. By the Step 1 above it suffices to see that $R_j\leq 0$, where
\begin{eqnarray} \nonumber
R_j&:=& \frac{1}{2}\left[ \frac{\frac{E_{j}+E_{j-1}}{2%
}(\beta x_{j})+\frac{E_{j}+E_{j-1}}{2}(\beta x_{j-1})}{2}-\frac{1}{h}%
\int_{\beta x_{j-1}}^{\beta x_{j}}\frac{E_{j}+E_{j-1}}{2}\,{\rm d}s\right] \\ \nonumber
&&+\frac{1}{2}\left[ \frac{\frac{E_{j}+E_{j-1}}{2}(\tau -\beta x_{j})+\frac{%
E_{j}+E_{j-1}}{2}(\tau -\beta x_{j-1})}{2}-\frac{1}{h}\int_{\tau
-\beta x_{j}}^{\tau -\beta
x_{j-1}}\frac{E_{j}+E_{j-1}}{2}\,{\rm d}s\right] \\ \label{eq_neg}
 &&+\frac{c_a}{4}\left[ \int_{\beta x_{j-1}}^{\beta x_{j}}
E_j(s)\,{\rm d}s + \int_{\tau-\beta x_{j}}^{\tau-\beta x_{j-1}} E_j(s)\,{\rm d}s
\right].
\end{eqnarray}%
In order to estimate $R_j$, remark that $E_j$ can be written as follows
$$E_j(t)=\left| \frac{u_{j+1}(t)-u_{j}(t)}{h}\right| ^{2}+\left|\frac{u_{j+1}^{\prime }(t)+
u_{j}^{\prime }(t)}{2}\right| ^{2}+a_M\left|
\frac{u_{j+1}(t)+u_{j}(t)}{2}\right| ^{2}=$$
$$= \left| b^n\frac{\phi^n_{j+1}- \phi^n_{j}}{h\sqrt{2 \mu_h^n}}e^{-i\lambda_h^n t}+b^m\frac{\phi^m_{j+1}-
\phi^m_{j}}{h\sqrt{2 \mu_h^m}}e^{-i\lambda_h^m t}\right| ^{2}$$ $$+
\left|b^n \mbox{\,sgn\,}(n) \frac{\phi^n_{j+1}+\phi^n_{j}}{2\sqrt{2}}
e^{-i\lambda_h^n t} + b^m \mbox{\,sgn\,}(m)
\frac{\phi^m_{j+1}+\phi^m_{j}}{2\sqrt{2}} e^{-i\lambda_h^m t}\right|
^{2}$$ $$+a_{M}\left|
b^n\frac{\phi^n_{j+1}+\phi^n_{j}}{2\sqrt{2 \mu_h^n}} e^{-i\lambda_h^n
t}+b^m \frac{\phi^m_{j+1}+\phi^m_{j}}{2\sqrt{2 \mu_h^m}}
e^{-i\lambda_h^m t}\right| ^{2}=
$$
$$
=\sum_{k=1}^3 \left| \alpha_{n,j}^k e^{-i\lambda_h^n
t}+\alpha_{m,j}^k e^{-i\lambda_h^m t}\right| ^{2},
$$
where, for $l\in\{n,m\}$, we set
$$\alpha_{l,j}^1= b^l \frac{\phi^l_{j+1}-\phi^l_{j}}{h\sqrt{2 \mu_h^l}},\quad
\alpha_{l,j}^2=b^l \mbox{\,sgn\,}(l)\frac{\phi^l_{j+1}+\phi^l_{j}}{2\sqrt{2}},\quad
\alpha_{l,j}^3=b^l a_M \frac{\phi^l_{j+1}+\phi^l_{j}}{2\sqrt{2 \mu_h^l}}.
$$

We deduce that $E_j$ is the
sum of three terms that satisfy the hypothesis of Lemma
\ref{le_c1}. From (\ref{desig}) with $r=\beta h$ and $t\in
\left\{\beta x_{j-1},\tau-\beta x_j \right\}$, we deduce
\begin{eqnarray} \nonumber
R_j&\leq & \frac{1-\beta}{2\beta h}
\left( \int_{\beta x_{j-1}}^{\beta x_{j}}\frac{E_{j}+E_{j-1}}{2}\,{\rm d}s
+ \int_{\tau -\beta x_{j}}^{\tau -\beta
x_{j-1}}\frac{E_{j}+E_{j-1}}{2}\,{\rm d}s\right) \\ \nonumber
&&+\frac{c_a}{4} \left[ \int_{\beta x_{j-1}}^{\beta x_{j}}
E_j(s)\,{\rm d}s + \int_{\tau-\beta x_{j}}^{\tau-\beta x_{j-1}} E_j(s)\,{\rm d}s
\right]\\ \nonumber
 &\leq & \left( \frac{1-\beta}{2\beta h}+\frac{c_a}{2}\right) \left( \int_{\beta x_{j-1}}^{\beta
x_{j}}\frac{E_{j}+E_{j-1}}{2}\,{\rm d}s + \int_{\tau -\beta x_{j}}^{\tau
-\beta x_{j-1}}\frac{E_{j}+E_{j-1}}{2}\,{\rm d}s\right). \label{eq_neg2}
\end{eqnarray}%
Note that \eqref{eq.gmic} implies that condition \eqref{eq:condgap} from Lemma \ref{le_c1} is verified.  From the fact that $\beta>1$  it follows that $R_j\leq 0$ for sufficiently small $h$ and \eqref{eq:step2} holds.

\bigskip

{\bf Step 3.} In this last step we show discrete versions of \eqref{eq:gr2} and \eqref{eq:gr3} for particular solutions of the form \eqref{eq.p}. The first one is the following discrete version of Gronwall's inequality that is easily deduced from the discrete inequality \eqref{eq:step2},
\begin{equation} \label{eq_dgi}
F_{j}\leq F_{0} \exp \left( \frac{2jhc_a}{2-hc_a}\right), \quad
1\leq j \leq N+1.
\end{equation}
Concerning the discrete version of \eqref{eq:gr3} we prove the following,
\begin{equation} \label{eq:step33}
\tilde E_h(0)\leq \frac{2}{\tau-2\beta}\exp \left(\frac{2c_a}{2-hc_a}\right)
\int_0^\tau \left[ \left( \frac{u_1(s)}{h}\right)^2 + \left(
\frac{u_1'(s)}{2}\right)^2 +a_M \left( \frac{u_1(s)}{2}\right)^2
\right] \,{\rm d}s,
\end{equation}
where $\tilde E_h(s)$ is defined by
\begin{equation}\label{eq:et}
\tilde E_h(s):=\frac12 \left[ \left\langle U'_h(s),U'_h(s)\right\rangle_M+\left\langle U_h(s),U_h(s)\right\rangle_1 \right]=h\sum_{j=0}^{N} \tilde E_j (s),
\end{equation}
and
\[
\tilde E_{j}(s):=\frac12 \left(\left|
\frac{u_{j+1}^{\prime }(s)+u_{j}^{\prime }(s)}{2}\right| ^{2}+\left| \frac{u_{j+1}(s)-u_{j}(s)}{h}\right| ^{2}+a_{j}|u_{j}(s)|^{2}\right)
\]%
Note that in \eqref{eq:et} we introduce a new energy $\tilde E_h(s)$, based on $\tilde E_j$ instead of $E_j$ defined in \eqref{eq:ej}, which has the conservation property: $$\tilde E_h(t)=\tilde E_h(0), \quad t\geq 0.$$ In fact, this is easily obtained from system \eqref{eq.dis0} by multiplying the $j-$equation by $u'_j$ and adding in $j=1,...,N$. On the other hand, since $\tilde E_j(s) \leq E_j(s)$, for $s\geq 0$, by using the conservation property of $\tilde E_h(s)$ and \eqref{eq_dgi}, we can   deduce a discrete version of \eqref{eq:gr3} as follows
\begin{eqnarray*}
(\tau-2\beta) \tilde E_h(0)&=&\int_{\beta}^{\tau-\beta} \tilde E_h(s) \,{\rm d}s=h\int_{\beta}^{\tau-\beta} \sum_{j=0}^{N} \tilde E_j (s) \,{\rm d}s \leq h\int_{\beta}^{\tau-\beta} \sum_{j=0}^{N} E_j (s) \,{\rm d}s \\
&\leq& h\sum_{j=0}^{N} \int_{\beta x_j}^{\tau-\beta x_j }  E_j (s) \,{\rm d}s= h\sum_{j=0}^{N} F_j \leq h \sum_{j=0}^{N} F_0  \exp \left( \frac{2jhc_a}{2-hc_a}\right) \\
&\leq &  F_0 \exp \left( \frac{2c_a}{2-hc_a} \right) .
\end{eqnarray*}
This completes the proof of \eqref{eq:step33} for particular solutions of the form \eqref{eq.p}. 

To finish the proof of Proposition \ref{prop:11} we observe that, due to the orthogonality and normalization of the eigenfunctions $\Phi_h^n$ in the norm $\|\cdot\|_{1,M}$ the left hand side in \eqref{eq:step33} reads,
\begin{eqnarray*}
\tilde E_h(0)&=&\frac12  \left[ \left\langle U'_h(0),U'_h(0)\right\rangle_M+\left\langle U_h(0),U_h(0)\right\rangle_1 \right]=\frac12  \left\| \left( \begin{array}{c} U_h(0) \\ U'_h(0) \end{array} \right) \right\|_{1,M}^2 \\&=&\frac12\left(|b^n|^2+|b^m|^2 \right),
\end{eqnarray*}
for particular solutions of the form \eqref{eq.p}. This
coincides with the left hand side in \eqref{eq:prop10b}.

On the other hand, the last term in the integrand of the right hand side in \eqref{eq:step33} can be bounded by the first one if $h$ is sufficiently small. Therefore, this term is bounded, up to a constant, by
$$
\int_0^\tau \left[ \left( \frac{u_1(s)}{h}\right)^2 + \left(
\frac{u_1'(s)}{2}\right)^2 
\right] \,{\rm d}s,
$$
which is the term on the right hand side in \eqref{eq:prop10b}. This completes the proof of Proposition \ref{prop:11} with a constant $C_\tau=\frac{2}{\tau-2\beta}\exp \left( \frac{2c_a}{2-hc_a} \right) $. 

\bigskip

We now turn our attention to property \eqref{eq:pp2}. We first give the following lemma.

\begin{lemma}\label{lemma:estgap} Let $\tau>2$ and $h>0$. Suppose that, given two different real numbers
$p>q$ and two nonzero complex numbers $\alpha_p,\alpha_q$, there
exits a constant $C_\tau>0$ such that we have
\begin{equation}\label{insim}
|a|^2+|b|^2\leq C_\tau \int_0^\tau \left[ \frac{1}{h^2}\left| a \alpha_p
 e^{ip t}+b \alpha_q e^{iq t}\right|^2 + \left| a \alpha_p
 p e^{i p t}+b \alpha_q q e^{iq
t}\right|^2 \right]\, \,{\rm d}t\quad (a,b\in\mathbb{C}).\end{equation}
Then, the following inequality holds
\begin{equation}\label{gaplm}
|p-q|\geq
\gamma:=\sqrt{\frac{\frac{1}{2C_\tau \tau }\left(\frac{1}{|\alpha_p|^2}+\frac{1}{|\alpha_q|^2}\right)}{1+\frac{1}{6}\tau^2
\left(\frac{1}{h^2}+2(|q|^2+|p|^2)\right)}}.
\end{equation}
\end{lemma}

\begin{proof} By taking $a\alpha_p=-b\alpha_q$, from \eqref{insim} we deduce that the following
relation holds true
$$
\int_0^\tau \left[\frac{1}{h^2}\left| 1-e^{i\, \zeta t}\right|^2
+\left|
 p - q e^{i\, \zeta t}\right|^2\right]\, \,{\rm d}t\geq \frac{1}{C_\tau}
 \left(\frac{1}{|\alpha_p|^2}+\frac{1}{|\alpha_q|^2}\right),
$$
where $\zeta=p-q$. Since $\zeta\in\mathbb{R}$, from the last
estimate it follows that
\begin{equation}\label{ze1}
|\zeta|^2 +
\left(\frac{1}{h^2}+2|q|^2\right)\left(1-\frac{\sin(\zeta
\tau)}{\zeta \tau}\right)\geq
\frac{1}{2C_\tau \tau}\left(\frac{1}{|\alpha_p|^2}+\frac{1}{|\alpha_q|^2}\right).
\end{equation}
By taking into account the inequality
$$1-\frac{\sin\, t}{t}\leq \frac{1}{6} t^2\qquad (t\geq 0)$$ we
deduce from \eqref{ze1} that \eqref{gaplm} holds.
\end{proof}

Now, we can formally state and give the proof of  \eqref{eq:pp2}.

\begin{proposition} \label{prop:12}
Assume that the hypotheses of Theorem  \ref{te:spectralgap} hold and let $\tau>2$. If $\lambda^n_h$ and $\lambda_h^m$ are two different eigenvalues of \eqref{eq.opmat} verifying \eqref{eq.gmic}, then \eqref{eq:pp2} is verified.
\end{proposition}

\begin{proof} By hypotheses, there exists $C_\tau>0$ such that \eqref{eq:prop10b} holds. From Lemma \ref{lemma:estgap} we obtain that
\begin{equation}\label{gaplm2}
|\lambda_h^n-\lambda_h^m|\geq
\sqrt{\frac{\frac{1}{2C_\tau \tau}\left(\frac{1}{(\phi^n_{1})^2}+\frac{1}{(\phi^m_{1})^2}\right)}{1+\frac{1}{6}\tau^2
\left(\frac{1}{h^2}+2\left(|\lambda_h^n|^2+|\lambda_h^m|^2\right)\right)}}.
\end{equation}
From Lemma \ref{te.obsing}, we have that, for $h\in(0,h_0)$,
the following relation is verified
\begin{equation}\label{gaplm3}\frac{1}{(\phi^n_{1})^2}+\frac{1}{(\phi^m_{1})^2}\geq \frac{2}{C_0}\left(\frac{2}{h^2}+|\lambda_h^n|^2
+|\lambda_h^m|^2\right).\end{equation}
From \eqref{gaplm2} and \eqref{gaplm3} we deduce that  \eqref{eq:pp2} holds and the proof of Proposition \ref{prop:12} is completed.\end{proof}

Finally, it remains to check property \eqref{eq:ing1}. This is easily obtained by considering a variant of the 
%
%
proof of Proposition \ref{prop:11} with particular solutions containing one single eigenfunction, i.e.
\begin{equation}\label{eq.p_s}
\left(\begin{array}{c}U_h\\U_h'\end{array}\right)(t)=
e^{-i\lambda_h^n t}\Phi^n_h.
\end{equation}

\begin{remark}\label{rem:top} We have proved that the uniform observability inequality \eqref{eq:obsineq} holds for any $T>T_0=\frac{2\pi}{\gamma_0}$, where $\gamma_0$ is given by \eqref{eq:gamma0}. Notice that, Proposition \ref{prop:11} shows that, when considering solutions of the discrete system with only two eigenmodes, the uniform observability inequality holds for any $\tau>2$, which represents the optimal time condition in the continuous wave equation. However,  Proposition \ref{prop:12}, based on inequality \eqref{eq:prop10b}, does not allow us to obtain the optimal value of the gap   $\gamma_0$ and, consequently, of the optimal time $T_0$.
\end{remark}

\section{Inverse source problem. Proof of Theorems \ref{te:estab0} and \ref{te:coninv}}

In this section we prove the two main results concerning the inverse source problem given in Section \ref{sec:2}.

\

\noindent{\it Proof of Theorem \ref{te:estab0}.} Due to the linearity of the problem, if $F_h=\widehat{F}_h-\widetilde{F}_h$, then $V_h=\widehat{V}_h-\widetilde{V}_h$ is the solution of system
\eqref{eq.us00}. Let us consider the solution $Z_h=\left[z_j\right]_{1\leq j\leq N}$ of the following homogeneous system  
\begin{equation}\label{eq.zh0}
\left\{
\begin{array}{l}
M_{h}Z_h''(t)+ K_h Z_h(t)+L_hZ_h(t)=0\qquad (t\in (0,T)),\\
Z_h(0)=0,\quad Z_h'(0)=F_h
\end{array}
\right.        
\end{equation}
and remark that 
\begin{equation}
V_h(t)=\int_0^t \lambda(t-s) Z_{h}(s)\,{\rm d}s \qquad (t\in(0,T)).
\end{equation}
If we denote $H^1_L(0,T;\mathbb{C})=\left\{ \zeta \in H^1(0,T;\mathbb{C})\, :\,  \zeta(0)=0 \right\}$, from the theory of Volterra integral operators (see, for instance,
Kress \cite[pp. 33–34]{kress}) it follows that there exists a constant $C>0$, independent of $h$, such that,
\begin{equation}\label{eq:inin1}
 \| \widehat{Y}_h-\widetilde{Y}_h \|_{[L^2(0,T)]^2}=\left\|  \int_0^t \lambda(t-s) \left(\begin{matrix}\smallskip \frac{z_1(s)}{h}\\  \frac{z_1'(s)}{2}\end{matrix}\right) \,{\rm d}s \right \|_{[H^1_L(0,T;\mathbb{C})]^2}   \geq C \left\| \left(\begin{matrix}\smallskip \frac{z_1}{h}\\ \frac{z_1'}{2}\end{matrix}\right)\right\|_{[L^2(0,T)]^2}.
\end{equation}
Since, according to Theorem \ref{te.obsineg},  the solution $Z_h$ of \eqref{eq.zh0} verifies the uniform observability inequality
$$
\left\| \left(\begin{matrix}\smallskip \frac{z_1}{h}\\ \frac{z_1'}{2}\end{matrix}\right)\right\|_{[L^2(0,T)]^2}^2\geq \kappa_0\, \|F_h\|_M^2,
$$
we deduce from \eqref{eq:inin1} that \eqref{eq:estab0} holds with a constant $\kappa=\frac{1}{C\kappa_0^{1/2}}$ independent of $h$.
\wfin

\begin{remark} Inequality \eqref{eq:estab0} in Theorem \ref{te:estab0} represents a stability result for the discrete inverse source problem with known intensity. It is worth mentioning that the constant $\kappa$ from inequality \eqref{eq:estab0} is independent of $h$ which is a direct consequence of the fact that the observability constant $\kappa_0$ in \eqref{eq:obsineq} has the same property. This will be of fundamental importance to demonstrate the convergence result from Theorem \ref{te:coninv} below. 
\end{remark}

\

\noindent{\it Proof of Theorem \ref{te:coninv}.} It is easy to see that ${\mathcal J}_h$ is continuous, strictly convex and coercive. Hence, ${\mathcal J}_h$ has a unique minimizer $\widehat{F}_h\in\mathbb{C}^N$.  

We consider the discretization $F_h=\left[f_j\right]_{1\leq j\leq N}$ of the source term $f$ given by \eqref{eq:fhdef} which satisfies the following convergence result
\begin{equation}\label{eq:cof}
f_h:= \sum_{j=1}^N f_{j}\psi_j \to f \mbox{ in }L^2(0,1) \mbox{ as }h\to 0.
\end{equation}
Let  $Y_h$ be the discrete observation corresponding to the solution $U_h$ of \eqref{eq.us00} with nonhomogeneous term $\lambda(t)M_h F_h$ and initial conditions $U_h^0$ and  $U_h^1$ given by \eqref{eq:datosdisc}. From Corollary \ref{cor:final} in the Appendix, it follows that 
\begin{equation}\label{eq:conv_nor_der}
\lim_{h\rightarrow 0}{\mathcal J}_h (F_h)=\frac{1}{2}\lim_{h\rightarrow 0} \left\| Y_h- y\right\|^2_{[L^2(0,T)]^2}=0.
\end{equation}
Since $\widehat{F}_h$ is a minimizer of ${\mathcal J}_h$, from \eqref{eq:conv_nor_der} we deduce that 
\begin{equation}\label{eq:conv_min}
\lim_{h\rightarrow 0}{\mathcal J}_h (\widehat{F}_h)=0.
\end{equation}

On the other hand, from \eqref{eq:estab0} in Theorem \ref{te:estab0}, we obtain that 
\begin{align*}
\|\widehat{f}_h-f_h\|_{L^2(0,1)}^2=\| \widehat{F}_h-F_h\|_M^2 &\leq \kappa^2 \| \widehat{Y}_h-Y_h \|_{[L^2(0,T)]^2}^2\leq 
4\kappa^2\left( {\mathcal J}_h (\widehat{F}_h)+{\mathcal J}_h (F_h)\right).
\end{align*}
Taking into account the independence of the constant $\kappa$ with respect to $h$, the above inequality, together with \eqref{eq:conv_nor_der} and \eqref{eq:conv_min}, implies that 
\begin{equation}\label{eq:conv_fh}
\lim_{h\rightarrow 0}\|\widehat{f}_h-f_h\|_{L^2(0,1)}=0.
\end{equation}
The above result and \eqref{eq:cof} show that the family $(\widehat{f}_h)_{h>0}$ converges to $f$ in $L^2(0,1)$ as $h$ tends to zero and the proof of the theorem is complete.
\wfin

\begin{remark}In the proof of Theorem  \ref{te:coninv}, one of the main ingredients is relation \eqref{eq:conv_nor_der} which states the convergence in $L^2(0,T)$ of the  normal derivative of the discrete solution to the continuous one, when the initial data is in the finite energy space $H_0^1(0,1)\times L^2(0,1)$. This type of result is not classical. It is known that, for the continuous problem with initial data in this energy space, the normal derivative is not more regular than $L^2(0,T)$. Moreover, this regularity it is not a consequence of the classical results of semigroup theory and it is usually referred to as ``hidden regularity''.  Therefore, the convergence \eqref{eq:conv_nor_der} represents a sharp result which needs a special treatment inspired in the multiplier techniques used for hyperbolic type equations. Its proof is given at the end of the Appendix in Theorem   \ref{te:codernor}.

\end{remark}

\section{Numerical experiments}

In this section we show some numerical experiments illustrating the convergence result of the inverse problem described above. To recover the source term form the boundary data we have implemented a classical gradient method to minimize the functional \eqref{eq:jn}. In particular we have taken  $\lambda \equiv 1$, the initial data zero, $T=3$ and the stop criterion is chosen when the gradient of the functional is lower than $10^{-6}$. 

\subsection{Smooth potential} \label{ssubsec:1}
In these experiments we consider that  the potential is given by $a(x)=1+0.5\sin(2\pi x)$.  Table \ref{table1} shows the $L^2-$norm of the difference between the source term and the reconstruction in two different situations, namely when the source term is a discontinuous function:
$$f(x)=20(x-1/2)^2\chi_{(1/4,3/4)}(x),$$ or a smooth function: $$g(x)=x(1-x)\left[5(x+0.1)^2+1/(x+0.1)\right],$$ respectively (see Figure \ref{fig:1}).  The three particular values chosen for the discretization parameter $h$ in Table \ref{table1} indicate at least a linear convergence rate of the method,  which becomes larger in the smoother case. 

\begin{figure}
    \centering
    \begin{tabular}{cc}
    \includegraphics[width=5cm]{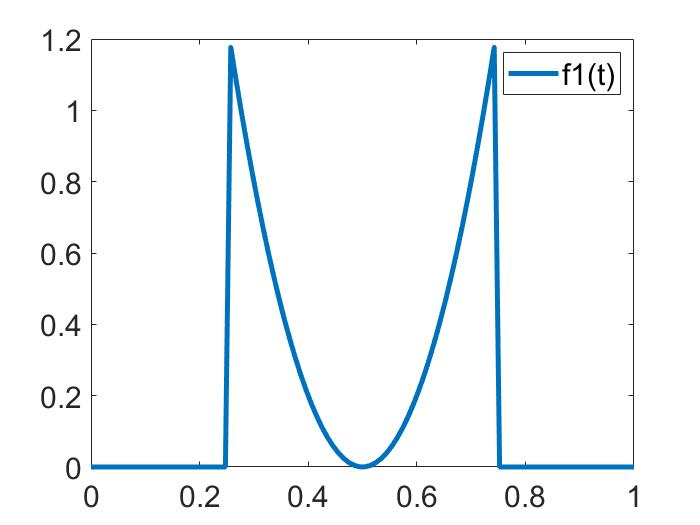} &
    \includegraphics[width=5cm]{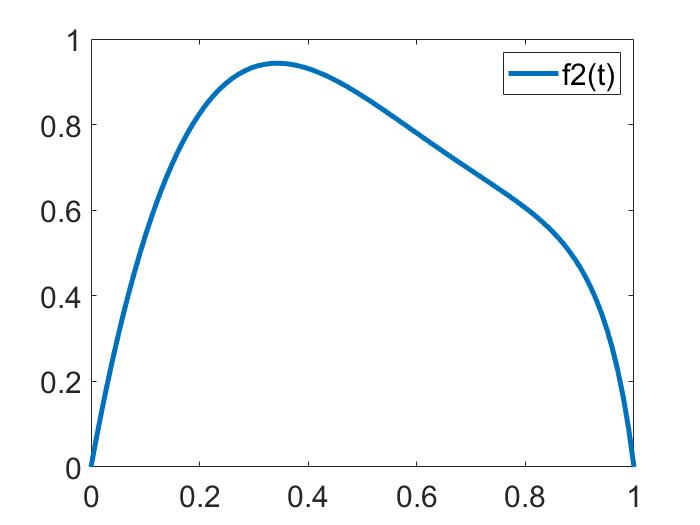} 
    \end{tabular}
    
    \caption{The two different source terms considered: a discontinuous one $f(x)$ (left) and a smooth one $g(x)$ (right).}
    \label{fig:1}
\end{figure}

\begin{table} \label{table:1}
\centering
\begin{tabular}{ |c|c|c| } 
				\hline
				$ h $ & $\| f-\widehat f_{h}\|_{L^2} $& $\| g - \widehat g_{h}\|_{L^2} $\\
				\hline
				$ 10^{-1} $	&$ 2.02\times10^{-1}  $& $ 3.13\times 10^{-2}  $\\ 
				$ 10^{-2} $ & $ 6.12\times10^{-2}  $ &  $ 7.09\times 10^{-4}  $\\
				$ 10^{-3} $&  $1.92\times10^{-2}$   & $ 4.05\times 10^{-6}$ \\ 
				\hline
    \end{tabular}
    \caption{Error in the numerical reconstruction of the source term for different values of $h$ and for two different source terms.}\label{table1}
\end{table}

\subsection{Noisy observation} To illustrate the stability of the method, we have considered three different noisy data. More precisely, instead of the observation  $y$ in \eqref{eq:jn}, we have taken
\begin{equation}\label{eq:nosydata}
y_{\delta}(t)=(1+\delta \zeta(t))y(t),
\end{equation}
where $\zeta(t)\in {\mathcal N}(0,1)$ and $\delta\in\{0.05,\,0.1,\,0.25\}$ gives  different values of relative error in the data. Here ${\mathcal N}(0,1)$ represents the normal distribution with zero mean and standard deviation one. The numerical results are summarized in Table \ref{table2} and show that the relative noise on the measurement has an influence on the error proportional
to its amplitude. This is a consequence of the uniform stability property proved in Theorem \ref{te:estab0}.

\begin{table} \label{table:2}
\centering
\begin{tabular}{ |c|c|c| c|} 
				\hline
				$ h $ & $\| f-\widehat f_{h,0.05}\|_{L^2} $& $\| f-\widehat f_{h,0.10}\|_{L^2} $& $\| f-\widehat f_{h,0.25}\|_{L^2} $\\
				\hline
				$ 10^{-1} $	&$ 1.96\times10^{-1}    $& $ 2.37\times 10^{-1} $&
                $ 2.72\times 10^{-1} $\\ 
				$ 10^{-2} $ & 
                $ 5.28\times10^{-2}  $ &  
                $ 7.53\times 10^{-2}$&$7.58\times10^{-2} $\\
				$ 10^{-3} $&  $2.00\times10^{-2} $  & $ 3.03\times 10^{-2} $  &
                $7.09 \times 10^{-2}$
                \\ 
				\hline
    \end{tabular}
    \caption{Error in the numerical reconstruction of the source term for different values of $h$ and for three different noisy source terms. Here $\widehat{f}_{h,\delta}$ is the numerical reconstruction of the source with each of the noisy data from \eqref{eq:nosydata}.}\label{table2}
\end{table}

\subsection{Discontinuous potential} 
We have tested the algorithm with a discontinuous potential function $a(x)=30\chi_{[0,0.5]}(x)+10\chi_{(0.5,1]}(x)$. Although the convergence of the method has been proved only for continuous potentials, the numerical results given in Table \ref{table3} are similar to those obtained in the continuous case.  This allows to conjecture that the convergence still holds for potentials with a finite number of discontinuities.

\begin{table} \label{table:3}
\centering
\begin{tabular}{ |c|c|c| } 
				\hline
				$ h $ & $\| f-\widehat f_{h}\|_{L^2} $& $\| g - \widehat g_{h}\|_{L^2} $\\
				\hline
				$ 10^{-1} $	&$ 1.62\times10^{-1}  $& $ 8.73\times 10^{-2}  $\\ 
				$ 10^{-2} $ & $ 4.38\times10^{-2}  $ &  $ 6.89\times 10^{-4}  $\\
				$ 10^{-3} $&  $1.42\times10^{-2}$   & $ 6.98\times 10^{-6}$ \\ 
				\hline
    \end{tabular}
    \caption{Error in the numerical reconstruction of the source term for different values of $h$ and for two different source terms in a case of discontinuous potential $a(x)=30\chi_{[0,0.5]}(x)+10\chi_{(0.5,1]}(x)$.}\label{table3}
\end{table}

\subsection{Smaller time observation} In Figure \ref{fig:2} we show how the reconstruction deteriorates when considering smaller observation times. For this experiment we have considered the source $g$ and the smooth potential $a$ from section \ref{ssubsec:1}. Note that for $T=2$, which is the minimal time for which both the continuous observability inequality \eqref{eq_observab} and the stability result \eqref{eq:stabcont} hold, the reconstruction is not accurate in the whole domain. This indicates that the minimal time $T_0$ required for the convergence of the numerical method in Theorem \ref{te:coninv} may be strictly larger than 2 (see Remark \ref{rem:top}).  
\begin{figure}[h]
    \centering
    \includegraphics[width=7cm]{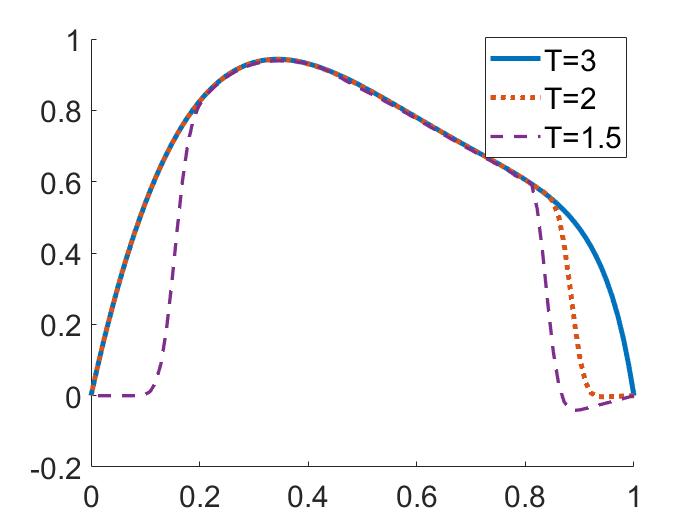} 
    \caption{Reconstruction of the source $g(x)$ for different time observations $T$.}
    \label{fig:2}
\end{figure}
\

\noindent \textbf{Acknowledgment:} The first author is supported by the Spanish Grant PID2021-124195NB-C31.

\

\section{Appendix}

In this appendix we show a series of convergence results for the mixed finite elements discretization method  introduced at the beginning of this paper. Given a non-negative potential $a\in C[0,1]$, we consider the one dimensional wave equation with Dirichlet boundary conditions,
\begin{equation}\label{eq.wave_ap12}
\left\{ \begin{array}{ll} w''(t,x)-w_{xx}(t,x)+a(x)w(t,x)= g(t,x)&
t\in(0,T),\,\, x\in(0,1)\\ w(t,0)=w(t,1)=0,\,\, & t\in(0,T)\\
w(0,x)=w^0(x),\,\, w'(0,x)=w^1(x)& x\in(0,1),\end{array}\right.
\end{equation}
where $(w^0,w^1)\in H^1_0(0,1)\times L^2(0,1)$ are initial data and $g\in L^2(0,T;L^2(0,1))$  is a nonhomogeneous term.  The corresponding variational formulation of  problem \eqref{eq.wave_ap1} reads as follows:
\begin{equation}\label{eq:varcont0}
\left\{\begin{array}{c}
w\in C\left( [0,T]; H_0^1(0,1)\right)\cap C^1 \left( [0,T]; L^2(0,1)\right),\\ \\
-\displaystyle \int_0^T  \left(w'(t), \varphi \right)_{L^2}p' (t)\,{\rm d}t + \int_0^T  \left( w(t),\varphi \right)_{H_0^1}p(t)\,{\rm d}t+\int_0^T \left( a w (t), \varphi \right)_{L^2}p (t)\,{\rm d}t\\ \\ =\displaystyle \int_0^T \left( g(t),\varphi \right)_{L^2}p (t)\,{\rm d}t\qquad \left(\varphi\in H^1_0(0,1) ,\quad p\in {\mathcal D}(0,T)\right),\\ \\
w(0)=w^0,\quad  w'(0)=w^1.
\end{array}
\right.
\end{equation}

It is known that, if $g\in L^2(0,T;L^2(0,1))$ and $(w^0,w^1)\in H^1_0(0,1)\times L^2(0,1)$, then \eqref{eq:varcont0} has a unique solution (see, for instance,  \cite[Chapter XVIII, Theorems 3, 4]{DL}). 

Now, we pass to study the space discretization of the above continuous problem. Firstly, we define the following orthogonal projection operators:
\begin{equation}\label{eq:proj1}
\begin{array}{l}
\mathbb{P}_1: H^1_0(0,1) \rightarrow X_{h,1}:=\mbox{Span}\{\varphi_1,\dots , \varphi_N\}\subset H^1_0(0,1), \\  \\ 
\mathbb{P}_2:L^2(0,1)\rightarrow X_{h,0}:=\mbox{Span}\{\psi_1,\dots , \psi_N\}\subset L^2(0,1),
\end{array}
\end{equation}
where $\left\{\varphi_j\right\}_{1\leq j\leq N}$ and $\left\{\psi_j\right\}_{1\leq j\leq N}$ are the two ``discrete basis functions'' from $H_0^1(0,1)$ and $L^2(0,1)$, respectively, introduced at the beginning of Section \ref{sec:2} to define the mixed finite elements method. We have that the following basic result holds.
\begin{proposition} \label{prop.csim} Given $w\in H_0^1(0,1)$ and $v\in L^2(0,1)$ we have that
\begin{equation}\label{eq:proj2}
\begin{array}{l}
\mathbb{P}_1 w= \displaystyle \sum_{j=1}^N w_j\varphi_j, \mbox{ with }w_j=w(x_j),
\\
\mathbb{P}_2 v= \displaystyle \sum_{j=1}^N v_j\psi_j, \mbox{ with }\frac{h}{4}\left(v_{j-1}+2v_j+v_{j+1}\right)=\frac{1}{2}\int_{x_{j-1}}^{x_{j+1}}v(x)\,{\rm d}x,
\end{array}
\end{equation}
where, by convention, $v_0=v_{N+1}=0$. Moreover, the following convergence results hold
\begin{equation}\label{eq:copr}
\begin{array}{c}\lim_{h\rightarrow 0}\mathbb{P}_1 w=w \mbox{ in }H_0^1(0,1),\\ \\
\lim_{h\rightarrow 0}\mathbb{P}_2 v=\lim_{h\rightarrow 0}\mathbb{P}_3 v=v \mbox{ in }L^2(0,1),\end{array}
\end{equation}where the operator $\mathbb{P}_3:L^2(0,1)\rightarrow X_{h,0}$ is defined by
\begin{equation}\label{eq:proj3}
\mathbb{P}_3 v= \displaystyle \sum_{j=1}^N \widetilde{v}_j\psi_j, \mbox{ with }\widetilde{v}_j=\frac{1}{2h}\int_{x_{j-1}}^{x_{j+1}}v(x)\,{\rm d}x,
\end{equation}
and it has the property that  $\|\mathbb{P}_3\|_{{\mathcal L}(L^2(0,1))}\leq 1$.
\end{proposition}
\begin{proof} The results are relatively well-known and we skip the proof.\end{proof}

On the other hand, we define three vectorial operators:
$$\widetilde{\mathbb{P}}_1:H_0^1(0,1)\to \mathbb{C}^N, \quad \widetilde{\mathbb{P}}_1w=W_h,$$ $$\widetilde{\mathbb{P}}_2:L^2(0,1)\to \mathbb{C}^N, \quad \widetilde{\mathbb{P}}_2v=h M_h^{-1}V_h,$$ 
$$\widetilde{\mathbb{P}}_3:L^2(0,1)\to \mathbb{C}^N, \quad \widetilde{\mathbb{P}}_3v=V_h,$$where, for each $j\in\{1,\dots,N\}$, the $j$-th component of the vectors $W_h=\left[w_j\right]_{1\leq j\leq N}$ and $V_h=\left[v_j\right]_{1\leq j\leq N}$ are given by $w_j=w(x_j)$ and  $\displaystyle v_{j}=\frac{1}{2h}\int_{x_{j-1}}^{x_{j+1}}v(x)\,{\rm d}x$, respectively.

\begin{remark}\label{rem:matconv0} Since the matrices $K_h$ and $M_h$ defined at the beginning of Section \ref{sec:2} are given by $\left[(\varphi_j,\varphi_k)_{H_0^1}\right]_{1\leq j,k\leq N}$ and $\left[(\psi_j,\psi_k)_{L^2}\right]_{1\leq j,k\leq N}$, respectively,  it follows that
\begin{equation}\label{eq:retip}
\left\langle K_h \widetilde{\mathbb{P}}_1w,\widetilde{\mathbb{P}}_1w\right\rangle=\|{\mathbb{P}}_1w \|_{H_0^1}^2,\qquad  \left\langle M_h \widetilde{\mathbb{P}}_2v,\widetilde{\mathbb{P}}_2v\right\rangle =\|{\mathbb{P}}_2v \|^2_{L^2},\end{equation}
for any $w\in H_0^1(0,1)$ and $v\in L^2(0,1)$. Moreover, according to  Proposition \ref{prop.csim}, it results that 
\begin{equation}\label{eq:cotip}\lim_{h\rightarrow 0} \left\langle \widetilde{\mathbb{P}}_3v,\widetilde{\mathbb{P}}_3v\right\rangle = \|v\|_{L^2}^2.\end{equation}
\end{remark}

Now, given $\begin{pmatrix}w^0\\w^1\end{pmatrix}\in H^1_0(0,1)\times L^2(0,1)$ and $g\in L^2(0,T;L^2(0,1))$, we consider their approximations:
\begin{equation}\label{eq:disdatain}
w_h^0=\mathbb{P}_1 w^0, \quad  w_h^1=\mathbb{P}_2 w^1, \quad g_h(t)=\mathbb{P}_2 g(t),\,\, t\in[0,T]. 
\end{equation}
Moreover, we define the vectors 
\begin{equation}\label{eq:disdatain2}
W^0_h= \widetilde{\mathbb{P}}_1w^0, \quad W^1_h =\widetilde{\mathbb{P}}_2 w^1,\quad G_h(t)=\widetilde{\mathbb{P}}_2 g(t) ,\,\, t\in[0,T].\end{equation} 
Notice that the components of the vector $W^0_h\in\mathbb{C}^N$ are the coefficients of the approximation $w_h^0$ in the basis $\{ \varphi_j \}_{j=1}^N$. Moreover, the vectors $W^1_h$ and $G_h(t)$ from $\mathbb{C}^N$ contain the coefficients of $w_h^1$ and $g_h(t)$ with respect to the basis $\{ \psi_j \}_{j=1}^N$, respectively. As explained in Section \ref{sec:2}, by using the mixed finite elements method, we are led to the following space semi-discrete version of   \eqref{eq.wave_ap1}:
\begin{equation}\label{eq.generalf}
\left\{\begin{array}{ll}
M_h W_h''(t)+K_h W_h(t)+L_h W_h(t)=M_h G_h(t) \qquad (t\in (0,T))\\ 
W_h(0)=W_h^0,\,\, W'_h(0)=W_h^1.
\end{array}
\right.
\end{equation}

Concerning the solutions of \eqref{eq.generalf}, we have the following result:
\begin{proposition} System \eqref{eq.generalf} has a unique solution $W_h=\left[w_j\right]_{1\leq j\leq N}\in  W^{1,\infty}\left(0,T; \mathbb{C}^N \right)\cap W^{2,2}\left(0,T; \mathbb{C}^N \right)$ with the following property
\begin{align} 
&\left\langle M_h W'_h(t),W_h'(t)\right\rangle+ \left\langle(K_h+L_h) W_h(t),W_h(t)\right\rangle\nonumber \\ &\leq C(T)\left[ \left\langle M_h W^1_h,W_h^1\right\rangle+ \left\langle(K_h+L_h) W_h^0,W_h^0\right\rangle + \int_0^T \left\langle M_h G_h(s),G_h(s)\right\rangle \,{\rm d}s\right], \label{eq:energyd}
\end{align}
where $C(T)$ is a positive constant depending only on $T$. Moreover, if $G_h=0$, then the following energy is conserved for $t\in[0,T]$:
\begin{equation}\label{eq:encon}
E_h(t):=\frac{1}{2}\left(\left\langle M_h W'_h(t),W_h'(t)\right\rangle+ \left\langle(K_h+L_h) W_h(t),W_h(t)\right\rangle\right)=\left\| \left(\begin{array}{c}W_h(t)\\W_h'(t)\end{array}\right)\right\|_{1,M}^2.
\end{equation}
\end{proposition}
\begin{proof} It is elementary and we omit it.\end{proof}
\begin{remark}\label{rem:poindis}By using the Fourier expansion of the solution $W_h(t)$ in the orthonormal basis consisting in the eigenvectors of the matrix $K_h$ we obtain that the following inequality holds:
\begin{equation}\label{eq:boh02}
\frac{4}{h}\sin^2\left(\frac{\pi h}{2}\right)\left\langle W_h(t),W_h(t)\right\rangle \leq \left\langle K_hW_h(t),W_h(t)\right\rangle \qquad (t\in[0,T]).
\end{equation}
Relation \eqref{eq:boh02} represents a discrete version on Poincar\'e inequality. Since $h\in (0,1)$, it follows that $\frac{4}{h^2}\sin^2\left(\frac{\pi h}{2}\right)\geq 4$ and consequently
\begin{equation}\label{eq:boh03}
4 h \left\langle W_h(t),W_h(t)\right\rangle \leq \left\langle K_hW_h(t),W_h(t)\right\rangle \qquad (t\in[0,T]).
\end{equation}
\end{remark}
If $W_h=\left[w_j\right]_{1\leq j\leq N}$ is the vectorial solution of \eqref{eq.generalf}, we introduce the functions
\begin{equation}\label{eq:df}
\left\{\begin{array}{l}w_h(t)= \displaystyle \sum_{j=1}^N w_j(t) \varphi_j \in H_0^1(0,1),\\ 
v_h(t)=  \displaystyle \sum_{j=1}^N w'_j(t) \psi_j \in L^2(0,1).
\end{array}\right.\end{equation}
With this notation, we have that the following relations hold: 
\begin{align}
K_{h,j}W_h(t)&=\frac{1}{h}\left(w_{j-1}(t)-2w_j(t)+w_{j+1}(t)\right)=\left( w_h(t), \varphi_j\right)_{H_0^1},\label{eq:uh}\\ 
M_{h,j}W'_h(t) &=\frac{1}{4}\left(w_{j-1}'(t)+2w'_j(t)+w'_{j+1}(t)\right)=\left( w'_h(t), \psi_j\right)_{L^2}=\left( v_h(t), \psi_j\right)_{L^2}.\label{eq:uhprim}
\end{align}
In \eqref{eq:uh}-\eqref{eq:uhprim} and in the sequel, $K_{h,j}$ and $M_{h,j}$ denote the $j$-th row of the matrices $K_h,\,M_h\in{\mathcal M}_N(\mathbb{C})$, respectively.  
It follows that  \eqref{eq.generalf} is equivalent to each of the following variational equations
\begin{equation}\label{eq:vardis00}
\left\{
\begin{array}{ll}
\left \langle M_hW''(t),\Theta_h\right\rangle+\left \langle K_hW(t),\Theta_h\right\rangle+\left\langle L_hW(t),\Theta_h\right\rangle =\left \langle M_h G_h(t),\Theta_h\right\rangle,\\ W_h(0)=W^0_h,\quad W'_h(0)=W^1_h,
\end{array}
\right.
\end{equation}
\begin{equation}\label{eq:vardis11}
\left\{
\begin{array}{ll}
\left\langle v_h'(t),\theta_h\right\rangle_{H^{-1},H_0^1}+\left(w_h(t),\theta_h\right)_{H_0^1}+\left(a w_h(t),\theta_h\right)_{L^2} =\left( g_h(t),\theta_h\right)_{L^2},\\  w_h(0)=w^0_h,\quad v_h(0)=w^1_h,
\end{array}
\right.
\end{equation}
for $t\in(0,T)$ and for each $\theta\in H_0^1(0,1)$, where $\Theta_h =\widetilde{\mathbb{P}}_1\theta \in \mathbb{C}^N$ and $\theta_h=\mathbb{P}_1\theta =\displaystyle \sum_{j=1}^N \Theta_{h,j} \varphi_j.$

The following result will be very useful in the proof of our convergence properties.
\begin{lemma}\label{lemma:matrice} Let $Q_h=[q_{ij}(h)]_{1\leq i,j\leq N}\in{\mathcal M}_N(\mathbb{C})$ be a tridiagonal matrix  ($q_{ij}(h)=0$ if  $|i-j|>1$) with the following additional properties:
\begin{enumerate}
\item $\max_{1\leq j\leq N}\left\{|q_{j j-1}(h)|,\, |q_{j j+1}(h)|\right\}\to 0 \mbox{ as }h\to 0,$
\item $\frac{1}{h}\max_{1\leq j\leq N}\left\{|q_{j j-1}(h)+q_{jj}(h)+q_{j j+1}(h)|\right\}\to 0 \mbox{ as }h\to 0$.
\end{enumerate}
If $Y_h=[y_j(h)]_{1\leq j\leq N}$ and $\Theta_h=[\theta_j(h)]_{1\leq j\leq N}$ are vectors from $\mathbb{C}^N$ with the property that there exists a positive constant $C$, independent of $h$, such that 
\begin{equation}\label{eq:no1max}\max\left\{ \left\langle K_h Y_h,Y_h\right\rangle,\, \left\langle K_h \Theta_h,\Theta_h\right\rangle\right\}\leq C,\end{equation}
then we have that 
\begin{equation}\label{eq:limax}
 \left\langle Q_h Y_h,\Theta_h\right\rangle \to 0 \mbox{ as }h\to 0.
\end{equation}
\end{lemma}
\begin{proof} Arguing as in Remark \ref{rem:poindis}, from \eqref{eq:no1max} we deduce that 
\begin{equation}\label{eq:no0max}
h \max\left\{ \left\langle Y_h,Y_h\right\rangle,\, \left\langle \Theta_h,\Theta_h\right\rangle\right\} \leq C.
\end{equation}
A straightforward computation shows that 
\begin{align*}
\left| \left\langle Q_h Y_h,\Theta_h\right\rangle\right|&=\left| \sum_{j=1}^N\left(q_{jj-1}(h)(y_{j-1}(h) -y_j(h))+q_{jj+1}(h)(y_{j+1}(h) -y_j(h))\right)\overline{\theta}_j(h)\right.\\&\left.+ \sum_{j=1}^N\left(q_{jj-1}(h)+q_{jj}(h)+q_{jj+1}(h)\right)y_j(h)\overline{\theta}_j(h)\right|\\ &\leq \max_{1\leq j\leq N}\left\{|q_{j j-1}(h)|,\, |q_{j j+1}(h)|\right\} \left(\left\langle K_h Y_h,Y_h\right\rangle+h \left\langle\Theta_h,\Theta_h\right\rangle\right) \\&+ \frac{1}{h}\max_{1\leq j\leq N}\left\{|q_{j j-1}(h)+q_{jj}(h)+q_{j j+1}(h)|\right\} \left(h\left\langle Y_h,Y_h\right\rangle+h \left\langle\Theta_h,\Theta_h\right\rangle\right).
\end{align*}
By taking into account properties 1. and 2. of the coefficients $q_{ij}(h)$  and the uniform boundedness relations \eqref{eq:no1max} and \eqref{eq:no0max}, from the above inequality it follows that \eqref{eq:limax} holds.
\end{proof}
\begin{remark}\label{rem:tresmat} Let $a\in C [0,1]$ and consider the following three matrices from ${\mathcal M}_N(\mathbb{R})$, $$S_h=\left[(a\varphi_j,\varphi_k)_{L^2}\right]_{1\leq j,k\leq N},\quad R_h=\left[(a\psi_j,\psi_k)_{L^2}\right]_{1\leq j,k\leq N},\quad L_h={\rm diag}\left(\left[ha(x_j)\right]_{1\leq j\leq N}\right).$$ Then the matrices $S_h-L_h$ and $R_h-L_h$ are tridiagonal and, under the hypothesis $a\in C [0,1]$, verify properties 1. and 2. from Lemma \ref{lemma:matrice}. Consequently, if $W_h(t)=\left[w_j(t)\right]_{1\leq j\leq N} $ is the solution of system \eqref{eq.generalf} (verifying the boundedness property \eqref{eq:energyd}) and  $\Theta_h=\widetilde{\mathbb{P}}_1\theta$ for some $\theta \in H_0^1(0,1)$, then the following properties hold uniformly on $t\in[0,T]$:
\begin{align}\label{eq:c3great}
 \left \langle S_hW_h (t),\Theta_h\right\rangle- \left \langle L_h W_h (t),\Theta_h\right \rangle \rightarrow 0 \mbox{ as }h\rightarrow 0,\\ \nonumber \\
\label{eq:c4great}
 \left\langle R_hW_h (t),\Theta_h\right\rangle- \left\langle L_h W_h (t),\Theta_h\right\rangle \rightarrow 0 \mbox{ as }h\rightarrow 0.
\end{align}
\end{remark}

\begin{remark} If in the above lemma we take $a\equiv 1$ and we take into account that in this particular case $R_h=M_h$,  we deduce that the following property holds uniformly on $t\in[0,T]$
\begin{equation}\label{eq:c1great}
 \left\langle M_hW_h (t),\Theta_h\right\rangle- \left\langle S_h W_h (t),\Theta_h\right\rangle \rightarrow 0 \mbox{ as }h\rightarrow 0.
\end{equation}
\end{remark}

Now we are ready to show the principal weak convergence result. 

\begin{proposition}[\bf Weak convergence of solutions]\label {prop:convweak} The family $(u_h)_{h>0}$ defined in \eqref{eq:df} is weakly* convergent in $W^{1,\infty}(0,T;L^2(0,1))\cap L^{\infty}(0,T;H^{1}_0(0,1))$ and weakly convergent in  $W^{2,2}(0,T;H^{-1}(0,1))$ to the unique solution  $u$ of \eqref{eq:varcont0}.
\end{proposition}

\begin{proof}  From \eqref{eq:energyd} and  \eqref{eq:vardis11} we deduce that the families $(w_h)_{h>0}$, $(v_h)_{h>0}$ and $(v_h')_{h>0}$ are uniformly bounded in $L^\infty\left(0,T;H_0^1(0,1)\right)$, $L^\infty\left(0,T;L^2(0,1)\right)$ and $L^2\left(0,T;H^{-1}(0,1)\right)$, respectively. It follows that there exist $w$, $v$ and $z$ such that (for some subsequences denoted in the same way)  we have 
\begin{equation}\label{eq:w*c}
\left\{\begin{array}{l}w_h \stackrel{\ast}{\rightharpoonup} w \mbox{ as }h\rightarrow 0 \mbox{ in }L^\infty\left(0,T;H_0^1(0,1)\right) ,\\ 
v_h(t) \stackrel{\ast}{\rightharpoonup} v \mbox{ as }h\rightarrow 0 \mbox{ in }L^\infty\left(0,T;L^2(0,1)\right),\\
v_h'(t)\rightharpoonup z \mbox{ as }h\rightarrow 0  \mbox{ in }L^2\left(0,T;H^{-1}(0,1)\right).
\end{array}\right.\end{equation}

Let $p\in{\mathcal D}(0,T)$, $\theta\in H_0^1(0,1)$, $\theta_h=\mathbb{P}_1\theta$ and $\Theta_h=\widetilde{\mathbb{P}}_1\theta$.  By taking into account  \eqref{eq:copr}, \eqref{eq:df} and \eqref{eq:w*c}, we obtain the following convergence properties as $h\to 0$:
\begin{equation}\label{eq:ch1}
\int_0^T \left\langle K_h W_h(t),\Theta_h\right\rangle p(t)\,{\rm d}t =\int_0^T \left(w_h(t),\theta_h\right)_{H_0^1} p(t)\,{\rm d}t \rightarrow \int_0^T \left(w(t),\theta \right)_{H_0^1} p(t)\,{\rm d}t.
\end{equation}
By taking into account the definition of the matrix $S_h$ from Remark \ref{rem:tresmat}, we obtain that \begin{equation}\label{eq:nuh2}
 \left\langle S_h W_h (t),\Theta_h\right\rangle=(a w_h(t),\theta_h)_{L^2}.
\end{equation}
Therefore, \eqref{eq:copr}, \eqref{eq:c3great} and \eqref{eq:w*c}  imply that 
\begin{equation}\label{eq:cal2}
\int_0^T \left\langle L_h W_h(t),\Theta_h\right\rangle p(t)\,{\rm d}t  \rightarrow \int_0^T \left(a w(t),\theta \right)_{L^2} p(t)\,{\rm d}t.
\end{equation}

On the other hand, we also have that 
\begin{align}\label{eq:cvl2}
&\int_0^T \left\langle M_hW_h'(t),\Theta_h\right\rangle p(t)\,{\rm d}t =\int_0^T \left(v_h(t),\theta_h\right)_{L^2}p(t)\,{\rm d}t \rightarrow \int_0^T \left(v(t),\theta \right)_{L^2} p(t)\,{\rm d}t.
\end{align}
By taking into account \eqref{eq:nuh2}, from  \eqref{eq:c1great} we deduce that 
\begin{equation}\label{eq:c2great}
\left\langle M_hW_h (t),\Theta_h\right\rangle \rightarrow (w(t),\theta)_{L^2}\mbox{ as }h\rightarrow 0,
\end{equation}
uniformly in $t\in[0,T]$, and therefore
\begin{align}\label{eq:cupl2}
&\int_0^T \left\langle M_hW_h'(t),\Theta_h\right\rangle p(t)\,{\rm d}t =-\int_0^T \left\langle M_hW_h (t),\Theta_h\right\rangle p'(t)\,{\rm d}t \rightarrow -\int_0^T \left(w(t),\theta \right)_{L^2} p'(t)\,{\rm d}t.
\end{align}
This shows that $w\in W^{1,\infty}(0,T;L^2(0,1))$ and $w'=v$. Moreover, the family $(w_h)_{h>0}$ is weakly* convergent to $w$ in $W^{1,\infty}(0,T;L^2(0,1)).$  Finally, we remark that
\begin{align}\label{eq:czhm1}
&\int_0^T \left\langle M_hW_h''(t),\Theta_h\right\rangle p(t)\,{\rm d}t =\int_0^T \left\langle v_h'(t),\theta_h\right\rangle_{H^{-1},H_0^1}p(t)\,{\rm d}t \rightarrow \int_0^T \left\langle z(t),\theta \right\rangle_{H^{-1},H_0^1} p(t)\,{\rm d}t,
\end{align}
which together with the following relation resulting again from \eqref{eq:c2great}, 
\begin{align}\label{eq:cupphmn}
&\int_0^T \left\langle M_hW_h''(t),\Theta_h\right\rangle p(t)\,{\rm d}t =\int_0^T \left\langle M_hW_h (t),\Theta_h\right\rangle p''(t)\,{\rm d}t \rightarrow \int_0^T \left(w(t),\theta \right)_{L^2} p''(t)\,{\rm d}t,
\end{align}
implies that $w\in W^{2,2}(0,T; H^{-1}(0,1))$, $w''=z$ and the family $(w_h)_{h>0}$ is weakly convergent to $w$ in $W^{2,2}(0,T;H^{-1}(0,1)).$ Since we also have that 
$$ \int_0^T \left\langle M_h G_h(t),\Theta_h\right\rangle p(t) \,{\rm d}t = \int_0^T \left(\mathbb{P}_2g(t),\theta_h\right)_{L^2}p(t) \,{\rm d}t \rightarrow  \int_0^T \left( g(t),\theta \right)_{L^2} p(t) \,{\rm d}t ,$$
it follows that $w\in W^{2,2}(0,T;H^{-1}(0,1))\cap W^{1,\infty}(0,T;L^{2}(0,1))\cap L^{\infty}(0,T;H^{1}_0(0,1))$ verifies the variational relation in \eqref{eq:varcont0}.

In order to show that $u$ verifies the initial conditions from \eqref{eq:varcont0}, we consider the variational formulation \eqref{eq:vardis00}, 
we multiply by a function $p\in C^2[0,T]$ with $p(0),\,p'(0)\neq 0$ and $p(T)=p'(T)=0$ and integrate by parts.  We obtain that 
\begin{multline}\label{eq:disin0}
\left\langle M_hW_h^1,\Theta_h\right\rangle p(0)-\left\langle M_hW_h^0,\Theta_h\right\rangle p'(0)\\ =\int_0^T \left[\left\langle M_hW_h(t),\Theta_h\right\rangle p''(t)+\left\langle (K_h+L_h)W_h(t),\Theta_h\right\rangle p(t)-\left\langle M_h G_h(t),\Theta_h\right\rangle p(t)\right]\,{\rm d}t.
\end{multline}
By passing to the limit as $h$ goes to zero in \eqref{eq:disin0} and integrating by parts, it follows that 
\begin{align*}
&\left( w^1,\theta\right)_{L^2}p(0) -  \left(w^0,\theta\right)_{L^2}p'(0)\\&=\int_0^T \left[\left(w(t),\theta\right)_{L^{2}}p''(t)+\left(w(t),\theta\right)_{H_0^1}p(t)+\left(a w(t),\theta\right)_{L^2}p(t)-\left( g(t),\theta\right)_{L^2}p(t)\right]\,{\rm d}t\\ &=-\left(w(0),\theta\right)_{L^2}p'(0)+\left( w'(0),\theta\right)_{L^2}p(0) \\ &+
\int_0^T \left[\left\langle w''(t),\theta\right\rangle_{H^{-1},H_0^1}+\left(w(t),\theta\right)_{H_0^1}+\left(a w(t),\theta\right)_{L^2}-\left( g(t),\theta\right)_{L^2}\right]p(t)\,{\rm d}t.
\end{align*}
From the above relation and  \eqref{eq:varcont0}, we deduce that $w'(0)=w^1$ and $w(0)=w^0$. 

According to \cite[Th\'eor\`eme 8.1, 8.2]{LM}, the limit function $w$ belongs to $C([0,T],H_0^1(0,1))\cap C^1([0,T],L^2(0,1))$, which concludes the proof of the proposition.

\end{proof}

\begin{remark}
In fact, the matrices $S_h$, $R_h$ and $L_h$ from Remark \ref{rem:tresmat} allow to give three possible discretizations of the potential term $a$ which are all convergent to the same limit. Indeed, if $W_h(t)=\left[w_j(t)\right]_{1\leq j\leq N} $ is the solution of system \eqref{eq.generalf} and  $\Theta_h=\widetilde{\mathbb{P}}_1\theta$ for some $\theta \in H_0^1(0,1)$,  from \eqref{eq:w*c}  and \eqref{eq:nuh2}  it follows that 
\begin{equation}\label{eq:c22great}
\left\langle S_hW_h (t),\Theta_h\right\rangle \rightarrow (a w(t),\theta)_{L^2},\mbox{ as }h\rightarrow 0.
\end{equation}
Now, by taking into account \eqref{eq:c22great},  \eqref{eq:c3great} and \eqref{eq:c4great} , we deduce that
$$
\left\langle L_hW_h (t),\Theta_h\right\rangle \rightarrow (a w(t),\theta)_{L^2}\mbox{ and }\left\langle R_hW_h (t),\Theta_h\right\rangle \rightarrow (a w(t),\theta)_{L^2},\mbox{ as }h\rightarrow 0.
$$
We could have chosen any of the above discretizations of the potential term, but the one involving the the matrix $L_h$ is the simplest and allowed us to show the observability inequality given  by Theorem \ref{te.obsineg}. Although in view of  \eqref{eq.waverew1} the most natural discretization of the potential term would be the one associated to the matrix $S_h$, the proof of the corresponding discrete observability inequality is more involved and needs further investigation.
\end{remark}

The following strong convergence result also holds.

\begin{proposition}[\bf Strong convergence of solutions]\label{prop.cosol} The families $(w_h)_{h>0}$ and $(v_h)_{h>0}$ defined in \eqref{eq:df}  verify the following relations as $h$ tends to zero:
\begin{equation}\label{eq:cosol}
w_h\rightarrow w \mbox{ in }L^2(0,T;H_0^1(0,1)) \mbox{ and } v_h\rightarrow w' \mbox{ in }L^2(0,T;L^2(0,1)),
\end{equation}
where  $w$  is the unique solution of \eqref{eq:varcont0}. \end{proposition}
\begin{proof}
Multiplying  \eqref{eq.generalf} by $W_h'$ and integrating by parts, we deduce that 
\begin{multline}
\left \langle M_h W_h'(t),W_h'(t)\right\rangle+\left\langle K_h W_h(t),W_h(t)\right\rangle+\left\langle L_h W_h(t),W_h(t)\right\rangle\\ =2\int_0^t \left\langle M_h G_h(s),W_h'(s)\right\rangle\,{\rm d}s + \left\langle M_h W_h^1,W_h^1\right\rangle+\left\langle K_h W_h^0,W_h^0\right\rangle+\left\langle L_h W_h^0,W_h^0\right\rangle,
\end{multline}
which is equivalent to 
\begin{multline}\label{eq:endisba}
\left(v_h(t),v_h(t)\right)_{L^2}+\left(w_h(t),w_h(t)\right)_{H_0^1}+\left(a w_h(t),w_h(t)\right)_{L^2}=2\int_0^t \left\langle M_h G_h(s),W_h'(s)\right\rangle\,{\rm d}s \\ + \left(v_h^0,v_h^0\right)_{L^2}+\left(w_h^0,w_h^0\right)_{H_0^1} + \left(a w_h^0,w_h^0\right)_{L^2}
+\left\langle(L_h-R_h) W_h^0,W_h^0\right\rangle-\left\langle(L_h-R_h) W_h(t),W_h(t)\right\rangle.
\end{multline}

On the other hand, the following relation is verified by the solution $u$ of \eqref{eq:varcont0}: 
\begin{multline}\label{eq:enconba}
\left(w'(t),w'(t)\right)_{L^2}+\left(w(t),w(t)\right)_{H_0^1}+\left(a w(t),w(t)\right)_{L^2}\\ =2\int_0^t \left( g(s),w'(s)\right)_{L^2}\,{\rm d}s + \left(w^1,w^1\right)_{L^2}+\left(w^0,w^0\right)_{H_0^1} + \left(a w^0,w^0\right)_{L^2}.
\end{multline}

By adding \eqref{eq:endisba} and \eqref{eq:enconba} we deduce that
\begin{align*}
&\frac{1}{2}\int_0^T\left(\left\|w'(t)-v_h(t)\right\|^2_{L^2}+\left\|w(t)-w_h(t)\right\|^2_{H_0^1}+\left(a w(t)-aw_h(t),w(t)-w_h(t)\right)_{L^2}\right)\,{\rm d}t\\&=-\int_0^T\left(\left(w'(t),v_h(t)\right)_{L^2}+ \left(w(t),w_h(t)\right)_{H_0^1}+\left(a w(t),w_h(t)\right)_{L^2}\right)\,{\rm d}t\\
&+\int_0^T\left(\int_0^t \left\langle M_h G_h(s),W_h'(s)\right\rangle\,{\rm d}s\right)\,{\rm d}t+\int_0^T\int_0^t \left(f(s),w'(s)\right)_{L^2}\,{\rm d}s \\
&+\frac{T}{2}\left(\left(v_h^0,v_h^0\right)_{L^2}+\left(w_h^0,w_h^0\right)_{H_0^1} + \left(a w_h^0,w_h^0\right)_{L^2}+ \left(w^1,w^1\right)_{L^2}+\left(w^0,w^0\right)_{H_0^1} + \left(a w^0,w^0\right)_{L^2}\right)\\&+\frac{T}{2}\left\langle (L_h-R_h) W_h^0,W_h^0\right\rangle-\int_0^T \left\langle (L_h-R_h) W_h(t),W_h(t)\right\rangle\,{\rm d}t.\end{align*}

The proposition is proved if we show that the right hand side tends to zero. This is a consequence of  \eqref{eq:enconba} and of the following convergence results: $v_h\rightharpoonup w'$ in $L^2(0,T;L^2(0,1))$,  $w_h\rightharpoonup w$ in $L^2(0,T;H_0^1(0,1))$ (from Proposition \ref{prop:convweak}),  $v^0_h\rightarrow w^1$ in $L^2(0,1)$, $w^0_h\rightarrow w^0$ in $H_0^1(0,1)$ (from Proposition \ref{prop.csim}), 
\begin{equation}\label{eq.conve10}
\frac{T}{2}\left\langle (L_h-R_h) W_h^0,W_h^0\right\rangle-\int_0^T \left\langle (L_h-R_h) W_h(t),W_h(t)\right\rangle\,{\rm d}t\rightarrow 0,
\end{equation}
(similarly to \eqref{eq:c4great}) and 
$$
\int_0^T \int_0^t \left\langle M_h G_h(s),W_h'(s)\right\rangle\,{\rm d}s \,{\rm d}t =\int_0^T \int_0^t\left(\mathbb{P}_2g(s),v_h(s)\right)_{L^2}\,{\rm d}s \,{\rm d}t   \rightarrow 
\int_0^T \int_0^t \left( g(s),w'(s)\right)_{L^2}\,{\rm d}s
$$
(which follows from \eqref{eq:copr} and Proposition \ref{prop:convweak}). The proof of the proposition is complete.
\end{proof}

\begin{remark}\label{rem:cpu} An argument similar to the one presented in the proof of Proposition \ref{prop.cosol} allows us to deduce the following pointwise in time convergence properties :
\begin{equation}
w_h(t)\rightarrow w(t) \mbox{ in }H_0^1(0,1) \mbox{ and } v_h(t)\rightarrow w'(t) \mbox{ in }L^2(0,1)\qquad (t\in[0,T]).
\end{equation}
\end{remark}

Finally, we give our main convergence result concerning the normal derivatives of the solutions.

\begin{theorem}[\bf Convergence of normal derivatives]\label{te:codernor} Assume that $W_h$ is the solution of \eqref{eq.generalf} with initial data $\begin{pmatrix}
W_{h}^0 \\
W_{h}^1
\end{pmatrix}$ and non-homogeneous term $G_h$ given by \eqref{eq:disdatain2}. If $w$  is the unique solution of \eqref{eq:varcont0}, then
\begin{equation}\label{eq:conor}
\frac{w_1(\,\cdot\,)}{h} \to w_x(0,\,\cdot\,) \mbox{ and } w_1'(\,\cdot\,) \to 0 \mbox{ in } L^2(0,T) \mbox{ as }h\rightarrow 0.
\end{equation}
\end{theorem}

\begin{proof} {\bf Step 1 (an identity):} In this first step we show that the following relation holds:
\begin{eqnarray}\nonumber
-\int_0^T\left(\frac{h}{8} \sum_{j=0}^N |w'_{j+1}+w'_{j}|^2 -
\frac{1}{8}|w'_1|^2\right)\,{\rm d}t -\int_0^T \left(\frac{1}{2h}\sum_{j=0}^N |w_{j+1}-w_{j}|^2
-\frac{1}{2h^2}|w_1|^2\right)\,{\rm d}t\\  \label{eq:car1}
+\Re (X_h(t)) \big{|}_0^T+ \int_0^T \Re \left(\sum_{j=1}^N (h a_j w_j - \widetilde{g}_j)(1-jh)\frac{\overline{w}_{j+1}-\overline{w}_{j-1}}{2h}\right)\,{\rm d}t=0,
\end{eqnarray}
where 
\begin{equation}\label{eq:xht}
X_h(t)= \frac{h}{4} \sum_{j=1}^N (w_{j-1}'+2w_j'+w_{j+1}')(1-hj)\frac{\overline{w}_{j+1}-\overline{w}_{j-1}}{2h}.
\end{equation}
In fact, \eqref{eq:car1} is easily deduced by multiplying the $j-$th equation in \eqref{eq.generalf} by the discrete multiplier $(1-hj)(\overline{w}_{j+1}-\overline{w}_{j-1})/(2h)$, adding in $j=1,...,N$ and integrating in time. 
For example, when considering the first term in \eqref{eq.generalf} we obtain,
\begin{align*}
& \int_0^T\Re\left[ \frac{h}4 \sum_{j=1}^N (w_{j-1}''+2w_j''+w_j'')(1-hj)\frac{\overline{w}_{j+1}-\overline{w}_{j-1}}{2h}\right] \,{\rm d}t = \Re (X_h(t)) \big{|}_0^T \\
& - \frac18 \int_0^T\Re\left[\sum_{j=1}^N \left((w_{j-1}'+w_j')+(w_{j}'+w_{j+1}') \right) \left((\overline{w}_{j}'+\overline{w}_{j+1}') -(\overline{w}_{j-1}'+\overline{w}_j')\right) (1-hj) \; \right]\,{\rm d}t \\
& =\Re (X_h(t)) \big{|}_0^T - \frac18 \int_0^T\sum_{j=1}^N \left[|w_{j+1}'+w_j'|^2 - |w_{j-1}'+w_j'|^2 \right] (1-hj)  \,{\rm d}t\\
& =\Re (X_h(t)) \big{|}_0^T -\int_0^T\left( \frac{h}8\sum_{j=0}^N |w_{j+1}'+w_j'|^2 - \frac{1}{8}|w_1|^2\right)\,{\rm d}t. 
\end{align*}
Analogously, for the second term in \eqref{eq.dis0},
\begin{align*}
&\int_0^T \Re\left[ \frac{1}h \sum_{j=1}^N (-w_{j-1}+2w_j-w_j)(1-hj)\frac{\overline{w}_{j+1}-\overline{w}_{j-1}}{2h}\right] \,{\rm d}t \\
& =\frac1{2h^2} \int_0^T\Re\left[\sum_{j=1}^N \left[ (w_{j}-w_{j-1})-(w_{j+1}-w_{j}) \right] \left[ (\overline{w}_{j}+\overline{w}_{j-1})+(\overline{w}_{j+1}+\overline{w}_{j}) \right] (1-hj) \right]\,{\rm d}t \\
& = \frac1{2h^2} \int_0^T\sum_{j=1}^N \left[|w_{j}-w_{j-1}|^2 - |w_{j+1}+w_j|^2 \right] (1-hj) \,{\rm d}t\\
& = \int_0^T\left(-\frac1{2h}\sum_{j=0}^N |w_{j+1}-w_j|^2 + \frac{1}{2h^2}|w_1|^2\right)\,{\rm d}t.
\end{align*}
We conclude that \eqref{eq:car1} holds. 

{\bf Step 2 (uniform boundedness):} From \eqref{eq:car1} a straightforward computation shows that there exists a constant $C>0$ such that  
\begin{equation} \label{eq_caso1}
\int_0^T \left(\left| \frac{w_1(t)}{h}\right|^2 +\left|\frac{w_1'(t)}{2}\right|^2 \right) \,{\rm d}t \leq C \left( \int_0^T \left(  a_M  E_h (s) +  \| G_h(t) \|_M^2 \right)\,{\rm d}s  +E_h (T)+E_h (0)\right) ,
\end{equation}
where $E_h$ is given by \eqref{eq:encon}.  For example, to estimate $X_h$ we combine Young's inequality and the following two estimates:
\begin{eqnarray*}
    && \frac{h}{4^2} \sum_{j=1}^N |w_{j-1}'+2w_{j}'+w_{j+1}'|^2 \leq \frac{h}{4^2} \sum_{j=1}^N 2\left[|w_{j-1}'+w_{j}'|^2+|w_{j}'+w_{j+1}'|^2\right] \leq \| W'\|^2_M.\\
    && h\sum_{j=1}^N \left| \frac{w_{j+1}-w_{j-1}}{2h} \right|^2 \leq h\sum_{j=1}^N \frac12 \left[\left| \frac{w_{j+1}-w_j}{h} \right|^2 + \left| \frac{w_{j}-w_{j-1}}{h} \right|^2\right]  \leq  \| W\|_1^2.
\end{eqnarray*}
From \eqref{eq_caso1}, taking into account \eqref{eq:energyd}, we easily deduce that there exists a constant $C>0$, independent of $h$, such that
\begin{equation} \label{eq:bound_c}
\int_0^T \left(\left| \frac{w_1(t)}{h}\right|^2 +\left|\frac{w_1'(t)}{2}\right|^2\right) \,{\rm d}t\leq C\left( \|(W_h^0,W_h^1)\|_{1,M}^2 + \int_0^T \| G_h(t) \|_M^2 \,{\rm d}t\right). 
\end{equation}

{\bf Step 3 (weak convergence):} From the previous step, it follows that there exists a subsequence, still denoted by $h$, such that 
$$
\frac{w_1(\,\cdot\,)}{h} \to z(\,\cdot\,)  \mbox{ weakly in }L^2(0,T) \mbox{ as } h \to 0,
$$
for some $z\in L^2(0,T)$. In this step we identify this limit and show that $z(t)=w_x(0,t)$.  

Multiplying the $j$-equation in \eqref{eq.generalf} by $(1-hj)p(t)$, where $p\in C_0^\infty(0,T)$, and adding in $j$ we easily obtain 
\begin{eqnarray*}
    0&=&\int_0^T\frac{h}4\sum_{j=1}^N (2w_j+w_{j+1}+w_j)(1-hj)p''(t)\,{\rm d}t +\int_0^T \frac{w_1}{h}p(t) \,{\rm d}t \\
    && +\int_0^T \left(h\sum_{j=1}^Na_jw_j(1-hj)-\sum_{j=1}^N \widetilde{g}_j(1-hj)\right) p(t) \,{\rm d}t \\
    &=& \int_0^T \langle M_h W_h ,Y_h\rangle \;  p''(t)\,{\rm d}t +\int_0^T \frac{w_1}{h}p(t)  \,{\rm d}t \\
    && + \int_0^T \left( \langle L_hW_h,Y_h\rangle - \langle M_h G_h,Y_h\rangle \right) p(t) \,{\rm d}t,
\end{eqnarray*}
where $Y_h=[1-x_j]_{1\leq j\leq N}$. Now, taking into account that 
\begin{eqnarray*}
    && \langle M_h W_h ,Y_h\rangle \to (w,1-x)_{L^2},\,   \langle L_hW_h,Y_h\rangle \to (aw,1-x)_{L^2},\, \langle M_h G_h,Y_h\rangle \to (g,1-x)_{L^2},
\end{eqnarray*}
as $h\to 0$ in $L^\infty(0,T)$, we obtain 
\begin{align} \nonumber
   \lim_{h\to 0} \int_0^T \frac{w_1}{h}p(t) \,{\rm d}t  &= -\int_0^T (w,1-x)_{L^2} p''(t) \,{\rm d}t - \int_0^T \left((aw,1-x)_{L^2}-(g,1-x)_{L^2} \right)p(t) \,{\rm d}t \\ \label{eq:we_co}
    & =\int_0^T w_x(t,0) \;  p(t) \,{\rm d}t,
\end{align}
 where the last equality is obtained by multiplying the equation of $w$ by $(1-x)p(t)$ and integrating. From \eqref{eq:we_co} we obtain the weak convergence of $w_1(\,\cdot\,)/h$ to $w_x(\,\cdot\, ,0)$ in $L^2(0,T)$.

{\bf Step 4 (convergence of norms):} In this step we prove that
\begin{equation}\label{eq:dnnorconv}
\left\|\frac{w_1(\,\cdot\,)}{h}\right\|^2_{L^2(0,T)} + \left\|\frac{w_1'(\,\cdot\,)}{2}\right\|^2_{L^2(0,T)} \to \| w_x(0,\,\cdot\,)\|_{L^2(0,T)}^2  \mbox{ as } h \to 0.
\end{equation}
Once this is proved, the lower semicontinuity of the norm with respect to the weak topology established in Step 3 above gives 
$$
 \| w_x(0,\,\cdot\,)\|_{L^2(0,T)}^2 \leq \liminf_{h\to 0}  \left(\left\|\frac{w_1(\,\cdot\,)}{h}\right\|^2_{L^2(0,T)}  + \left\|\frac{w_1'(\,\cdot\,)}{2}\right\|^2_{L^2(0,T)} \right) = \| w_x(0,\,\cdot\,)\|_{L^2(0,T)}^2 . 
$$
This means in particular,
\begin{eqnarray*}
    && \lim_{h\to 0} \left\|\frac{w_1(\,\cdot\,)}{h}\right\|^2_{L^2(0,T)}  = \| w_x(0,\,\cdot\,)\|_{L^2(0,T)}^2,\quad 
      \lim_{h\to 0} \left\|\frac{w_1'(\,\cdot\,)}{2}\right\|^2_{L^2(0,T)} =0.
\end{eqnarray*}
This norm convergence, together with the result from Step 3, allow to deduce that \eqref{eq:conor} holds and the proof of the Proposition is complete.  Therefore, it remains to prove that \eqref{eq:dnnorconv} is verified.  Note that in \eqref{eq:car1} we already have two terms representing the $L^2(0,T)-$norms of $w_1/h$  and $w_1'/2$. We only have to analyze the convergence of the remaining terms as $h\to 0$. 

From Proposition \ref{prop.cosol} we deduce that 
\begin{align}\label{eq:comul1}
&\int_0^T \left(\frac{h}{4}\sum_{j=0}^N (w_{j+1}'+w_{j}')^2\right) \,{\rm d}t = \int_0^T \|v_h(t)\|^2_{L^2}\,{\rm d}t \rightarrow \int_0^T\|w'(t)\|_{L^2}^2\,{\rm d}t\mbox{ as }h\rightarrow 0,\\
&\int_0^T \left(\frac{1}{h}\sum_{j=0}^N (w_{j+1}-w_{j})^2\right) \,{\rm d}t=\int_0^T \|w_h(t)\|^2_{H_0^1}\,{\rm d}t \rightarrow \int_0^T\|w(t)\|_{H_0^1}^2\,{\rm d}t\mbox{ as }h\rightarrow 0.
\end{align}

Concerning the term $X_h(t)$ for $t=0,T$, we prove that 
\begin{equation}\label{eq:comul2}
X_h(t) \to \int_0^1 w'(t,x)(1-x)\overline{w}_x(t,x)\,{\rm d}x\mbox{ as }h\rightarrow 0.
\end{equation}
Firstly, by using \eqref{eq:proj3} and \eqref{eq:df}, we easily deduce that the following relation is verified:
\begin{equation}\label{eq:fxht}X_h(t) -\int_0^1 \mathbb{P}_3v_h(t,x)(1-x)(\overline{w}_{h})_x(t,x) \,{\rm d}x=-h \left\langle M_h W_h'(t),K_h W_h(t)\right\rangle.\end{equation}
Since, for each $t\in[0,T]$,  we have that
\begin{equation}\label{eq:likein} \left| h\left\langle M_h W_h'(t),K_h W_h(t)\right\rangle  \right| \leq h\|M_h\|^\frac{1}{2} \| M_h^\frac{1}{2}W_h'(t)\| \|K_h\|^{\frac{1}{2}}\|K_h^\frac{1}{2} W_h(t)\| ,\end{equation}
from the energy estimate \eqref{eq:energyd} and the definition of the matrices $K_h$ and $M_h$  
we obtain that the right hand side term in \eqref{eq:fxht} tends to $0$ as $h\to 0$. On the other hand, taking into account the convergence results from Proposition \ref{prop.csim} and Remark \ref{rem:cpu}, it follows  that  \eqref{eq:comul2} results by passing to the limit in \eqref{eq:fxht}. 

Now we pass to evaluate the potential term in \eqref{eq:car1}. We remark that \eqref{eq:df} allows us to deduce the following relations
\begin{align}
&\left|\int_0^1 a(x) (1-x) w_h(t,x) \overline{w}_{h,x}(t,x)\,{\rm d}x - \sum_{j=1}^N a_j w_j(t)(1-jh) \frac{\overline{w}_{j+1}(t)-\overline{w}_{j-1}(t)}{2}\right|\nonumber \\ & =\left|\sum_{j=1}^N \widetilde{q}_{jj-1}\left(\overline{w}_{j-1}(t)-\overline{w}_{j}(t)\right) w_j(t) +\sum_{j=1}^N \widetilde{q}_{jj+1}\left(\overline{w}_{j}(t)-\overline{w}_{j+1}(t)\right) w_j(t)\right| \nonumber  \\ & \leq h\max_{1\leq j\leq N}\left\{ |\widetilde{q}_{j j+1}(h)|,|\widetilde{q}_{jj-1}(h)|\right\} \sum_{j=1}^N \left(\left|\frac{\overline{w}_{j-1}(t)-\overline{w}_{j}(t)}{h}\right|^2 +\left| w_j(t) \right|^2\right),\label{eq:amu}
\end{align}
where $$\begin{array}{cl}\widetilde{q}_{jj-1}(h)= 
-\frac{1}{h^2}\int_{x_{j-1}}^{x_j}(1-x) a(x) (x-x_{j-1})\,{\rm d}x +\frac{1}{2}a_j (1-x_j), \\ \\
\widetilde{q}_{jj+1}(h)=\frac{1}{h^2}\int_{x_{j}}^{x_{j+1}}(1-x) a(x) (x_{j+1}-x)\,{\rm d}x-\frac{1}{2}a_j (1-x_j).
\end{array}$$
Since $a\in C[0,1]$, we deduce that, for any $\varepsilon>0$, there exists $h_\varepsilon>0$ such that 
\begin{equation}\label{eq:lia1}
\max_{1\leq j\leq N}\left\{ |\widetilde{q}_{j j+1}(h)|,|\widetilde{q}_{jj-1}(h)|\right\}<\varepsilon\qquad (h\in(0,h_\varepsilon)).
\end{equation}
Taking into account \eqref{eq:energyd}, \eqref{eq:boh03} and \eqref{eq:lia1},  we deduce  that the right hand side term in \eqref{eq:amu} vanishes as $h\to 0$. Proposition \ref{prop.cosol} implies  that the following convergence result holds as $h$ tends to $0$:
\begin{align}
\int_0^T \sum_{j=1}^N a_j w_j(t)(1-jh) \frac{\overline{w}_{j+1}(t)-\overline{w}_{j-1}(t)}{2}\,{\rm d}t  \rightarrow \int_0^T\int_0^1 a(x) (1-x)w(t,x)\overline{w}_x(t,x)\,{\rm d}x\,{\rm d}t.\label{eq:comul3}
\end{align}
Finally, let us remark that, according to Proposition \ref{prop.csim}, the family $\left(\widetilde{g}_h(t)\right)_h$, defined by $$\widetilde{g}_h(t)=\mathbb{P}_3g= \sum_{j=1}^N \frac{1}{h}\widetilde{g}_j(t) \psi_j,\qquad \widetilde{g}_j(t)=\frac{1}{2}\int_{x_{j-1}}^{x_{j+1}}g(t,x)\,{\rm d}x  \quad \left(1\leq j\leq N\right),$$ converges to $g(t)$ in $L^2(0,1)$. Moreover, we have that 
\begin{align*}
&\int_0^1 \widetilde{g}_h(t,x) (1-x) \overline{u}_{h,x}(t,x)\,{\rm d}x - \sum_{j=1}^N \widetilde{g}_j(t) (1-jh) \frac{\overline{w}_{j+1}(t)-\overline{w}_{j-1}(t)}{2h}\\ &=\sum_{j,k=1}^N \frac{\widetilde{g}_j(t)}{h}\overline{w}_k(t)  \int_0^1(1- x) \psi_{j}(x)\varphi_{k,x}(x)\,{\rm d}x - \sum_{j=1}^N \widetilde{g}_j(t)(1-jh) \frac{\overline{w}_{j+1}(t)-\overline{w}_{j-1}(t)}{2h}\\& =\sum_{j=1}^N \frac{\widetilde{g}_j(t)}{2h}\left(\left(-(1-jh)-\frac{h}{2}\right)\overline{w}_{j-1}(t) +h\overline{w}_{j}(t)+\left((1-jh)-\frac{h}{2}\right)\overline{w}_{j+1}(t)  \right)\\ &- \sum_{j=1}^N \widetilde{g}_j(t)(1-jh) \frac{\overline{w}_{j+1}(t)-\overline{w}_{j-1}(t)}{2h} =\frac{h}{4}\left\langle M_h G_h(t),K_h W_h(t)\right\rangle.
\end{align*}
By using the same argument as in \eqref{eq:likein},  we obtain that the last term  above vanishes as $h$ goes to zero. From Proposition \ref{prop.cosol} it results that the following convergence result holds as $h$ tends to $0$:
\begin{align}
&\int_0^T  \sum_{j=1}^N \widetilde{g}_j(t) w_j(t)(1-jh) \frac{\overline{w}_{j+1}(t)-\overline{w}_{j-1}(t)}{2h} \,{\rm d}t \rightarrow \int_0^T\int_0^1 g(t,x) (1-x)\overline{u}_x(t,x)\,{\rm d}x\,{\rm d}t.\label{eq:comul4}
\end{align}
From the following relation verified by the solution of the wave equation
\begin{align*}
&-\frac{1}{2}\int_0^T \left( \|w'(t)\|_{L_2}^2 +\|w(t)\|_{H_0^1}^2\right)\,{\rm d}t +\frac{1}{2}\int_0^T |w_x(t,0)|^2 \,{\rm d}t +\left.\Re \int_0^1 w'(t,x)(1-x)\overline{w}_x(t,x)\right|_0^T\\&+\int_0^T\int_0^1\Re\left[ (a(x) w(t,x)- g(t,x))(1-x)\overline{w}_x(t,x)\right]\,{\rm d}x\,{\rm d}t =0,
\end{align*}
by using \eqref{eq:car1}, \eqref{eq:comul1},  \eqref{eq:comul2},  \eqref{eq:comul3} and  \eqref{eq:comul4} we deduce that \eqref{eq:dnnorconv} holds true. 
\end{proof}

We end this section with the convergence result for the discrete observation.
\begin{corollary}\label{cor:final} Let $\begin{pmatrix}w^0\\w^1\end{pmatrix}\in H^2 (0,1)\cap H^1_0(0,1)\times H^1_0(0,1)$ and  $y$ be the observation \eqref{eq:obscon00} of the continuous inverse problem \eqref{eq:inv_cont0}. If $Y_h$ is the discrete observation defined by \eqref{eq:obsdis00}, where  $F_h=\widetilde{\mathbb{P}}_2 f$ and $\begin{pmatrix}U_h^0\\U_h^1\end{pmatrix}$  are given by \eqref{eq:datosdisc}, the following convergence result holds:
\begin{equation}\label{eq:conv_nor_der_final} 
\lim_{h\rightarrow 0} \left\| Y_h- y\right\|^2_{[L^2(0,T)]^2}=0.
\end{equation}
\end{corollary}
\begin{proof}Firstly, let us recall that the time derivative of the solution $w$ of equation \eqref{eq:inv_cont0} can be written as follows:
$$w'(t)=v'(t) +u(t),$$
where $v$ and $u$ are the solutions of  \eqref{eq:inv_cont01} and \eqref{eq:inv_cont02}, respectively. Consequently, the boundary observation of the continuous inverse problem \eqref{eq:inv_cont0} is given by  \begin{equation}\label{eq:obs_cont_inv} y(t)=\left(\begin{matrix} v_x'(t,0)+u_x(t,0)\\0\end{matrix}\right)\quad  (t\in(0,T)).\end{equation}  

Firstly, if  $U_h=\left[u_j\right]_{1\leq j\leq N}$ is the solution of \eqref{eq.us000}, from Theorem \ref{te:codernor} and \eqref{eq:convdatosdisc} we deduce that 
\begin{equation}\label{eq:convno1}
\begin{pmatrix}\smallskip\frac{1}{h}u_{1}(\,\cdot\, ) \\  \frac{1}{2}u_{1}'(\,\cdot\, )\end{pmatrix} \rightarrow \begin{pmatrix}\smallskip u_x(\cdot\, , 0) \\  0\end{pmatrix} \mbox{ in }[L^2(0,T)]^2.
\end{equation}

On the other hand, the solution $V_h=\left[v_j\right]_{1\leq j\leq N}$ of \eqref{eq.us00} verifies
\begin{equation}\label{eq:exp1}
V_h'(t)= \lambda(0) Z_h(t) +\int_0^t \lambda'(t-s) Z_{h}(s)\,{\rm d}s \qquad (t\in(0,T)),
\end{equation}
where $Z_h=\left[z_j\right]_{1\leq j\leq N}$ is the solution of  \eqref{eq.zh0}. By using again Theorem \ref{te:codernor} and \eqref{eq:convnonhomterm}, we deduce that 
\begin{equation}
\begin{pmatrix}\smallskip\frac{1}{h}z_{1}(\,\cdot\, ) \\  \frac{1}{2}z_{1}'(\,\cdot\, )\end{pmatrix} \rightarrow \begin{pmatrix}\smallskip z_x(\cdot\, , 0) \\  0\end{pmatrix} \mbox{ in }[L^2(0,T)]^2,\end{equation}
where $z$ is the solution of 
\begin{equation}\label{eq:duhamel}
\left\{ \begin{array}{ll} z''(t,x)- z_{xx}(t,x)+a(x) z(t,x)= 0&
t\in(0,T),\,\, x\in(0,1)\\  z(t,0)=  z(t,1)=0,\,\, & t\in(0,T)\\  z(0,x)=0,\,\,  z'(0,x)=f(x)& x\in(0,1).\end{array}\right.
\end{equation}
Since $$
v'(t)= \lambda(0) z(t) +\int_0^t \lambda'(t-s) z(s)\,{\rm d}s \qquad (t\in(0,T)),
$$ from \eqref{eq:exp1} we deduce that 
\begin{equation}\label{eq:convno2}
\begin{pmatrix}\smallskip\frac{1}{h}v'_{1}(\,\cdot\, ) \\  \frac{1}{2}v_{1}''(\,\cdot\, )\end{pmatrix} \rightarrow \begin{pmatrix}\smallskip v'_x(\cdot\, , 0) \\  0\end{pmatrix} \mbox{ in }[L^2(0,T)]^2.
\end{equation}
The convergence \eqref{eq:conv_nor_der_final}  follows from \eqref{eq:convno1} and \eqref{eq:convno2}.
\end{proof}

\end{document}